\documentclass[11pt,reqno]{amsart}
\usepackage{latexsym,amsmath,amssymb, amsthm, mathscinet}
\usepackage{cases, verbatim}
\usepackage[dvipdfmx]{graphicx}
\usepackage{bmpsize}
\usepackage[active]{srcltx}
\usepackage{float}
\usepackage{hyperref}

\usepackage{amsfonts,amsmath,amsthm, amssymb}
\usepackage{cases}
\usepackage{enumerate} 

\theoremstyle{plain}
\newtheorem{thm}{Theorem}[section]

\newtheorem{defn}{Definition}[section]

\newtheorem{lem}[thm]{Lemma}
\newtheorem{cor}[thm]{Corollary}

\newtheorem{prop}[thm]{Proposition}

\theoremstyle{remark}
\newtheorem{rem}{\bf{Remark}}[section]
\newtheorem{ex}{Example}[section]

\numberwithin{equation}{section}

\newcommand{\N}{\mathbb{N}}

\newcommand{\R}{\mathbb{R}}

\newcommand{\cN}{\mathcal{N}}

\newcommand{\ocC}{\overline{\mathcal{C}}}

\newcommand{\USC}{{\rm USC\,}}
\newcommand{\LSC}{{\rm LSC\,}}

\newcommand{\Lip}{{\rm Lip\,}}

\newcommand{\ep}{\varepsilon}

\newcommand{\ol}{\overline}
\newcommand{\ul}{\underline}

\newcommand{\inter}{{\rm int}\,}
\newcommand{\cl}{{\rm cl}\,}

\newcommand{\dist}{{\rm dist}\,}

\newcommand{\sd}{{\rm sd}}

\newcommand{\Per}{{\rm Per}\,}
\newcommand{\vep}{\varepsilon}

\newcommand{\re}{\textcolor{red}}
\newcommand{\bl}{\textcolor{blue}}
\newcommand{\ma}{\textcolor{magenta}}

\begin{document}
\title[A representation formula for nonlocal equations]{A representation formula for viscosity solutions of nonlocal Hamilton--Jacobi equations and applications}

\date{\today}

\author[T. Kagaya]{Takashi Kagaya}
\address[T. Kagaya]{Graduate school of Engineering, Muroran Institute of Technology, 27-1 Mizumoto-cho, Hokkaido, 050-8585, Japan}
\email{kagaya@mmm.muroran-it.ac.jp}

\author[Q. Liu]{Qing Liu}
\address[Q. Liu]{Geometric Partial Differential Equations Unit, Okinawa Institute of Science and Technology Graduate University, 1919-1 Tancha, Onna-son, Kunigami-gun, 
Okinawa, 904-0495, Japan, 
}
\email{qing.liu@oist.jp}

\author[H. Mitake]{Hiroyoshi Mitake}
\address[H. Mitake]{
Graduate School of Mathematical Sciences, 
University of Tokyo, 
3-8-1 Komaba, Meguro-ku, Tokyo, 153-8914, Japan}
\email{mitake@ms.u-tokyo.ac.jp}

\date{\today}

\begin{abstract}
This paper is concerned with geometric motion of a closed surface whose velocity depends on a nonlocal quantity of 
the enclosed region.  Using the level set formulation, we study a class of nonlocal Hamilton--Jacobi equations and establish a control-based representation 
formula for solutions. We also apply the formula to discuss the fattening phenomenon and large-time asymptotics of the solutions. 
\end{abstract}

\maketitle

\section{Introduction}
In this paper we study the Cauchy problem for a class of nonlocal Hamilton--Jacobi equations of the form 
\begin{numcases}{}
u_t(x,t)+|\nabla u(x, t)|f\big(u(x,t), \mu(\{u(\cdot, t) < u(x, t)\})\big)=0 & in $ \mathbb{R}^n\times(0, \infty)$, \label{E1}\\
u(x,0)=u_0(x) & in $ \mathbb{R}^n$,  \label{initial}
\end{numcases}
where an initial condition $u_0: \R^n\to \R$ is continuous. 
Here, $f: \R \times [0, \infty) \to \R$ and $\mu: \mathcal{B} \to [0, \infty)$ are given functions, where we denote by $\mathcal{B}$ 
the family of all measurable sets in $\R^n$. We impose the following assumptions on $f$ and $\mu$. 
\begin{enumerate}
\item[(A1)] $f$ is bounded and continuous. Furthermore, there exists a modulus of continuity $\omega_f$ such that 
\[ |f(r, q_1) - f(r, q_2)| \le \omega_f (|q_1- q_2|) \quad \text{for} \; \; q \in \R, \; q_1, q_2 \ge 0. \]
\item[(A2)] $f$ is increasing in each variable in the following sense: 
\[
f(r_1, q) \le f(r_2, q) \quad \text{for} \; \; r_1 < r_2, \; q \ge 0,
\]
\begin{equation}\label{monotone eq}
f(r, q_1) < f(r, q_2) \quad \text{for} \; \; r \in \R, \; 0 \le q_1 < q_2. 
\end{equation}
\item[(A3)] $\mu$ is a measure in $\R^n$ given by 
\[ \mu(A) := \int_{A} \Theta(x) \; dx \quad \text{for} \; \; A \in \mathcal{B} \]
for a function $\Theta\in L^1(\R^n)$ with $\Theta>0$ in $\R^n$. 
\end{enumerate}
The condition \eqref{monotone eq} in (A2), together with (A3), implies the so-called monotonicity of the nonlocal Hamiltonian. Recall that a general first order nonlocal operator $F: \R\times \R^n\times \mathcal{B}\to \R$ is said to be monotone if $F(r, p, A)\leq F(r, p,  B)$ for all $r\in \R$, $p\in \R^n$ and $A, B\in \mathcal{B}$ satisfying $A\subset B$. 

For the initial value $u_0$, we assume that 
\begin{itemize}
\item[(A4)]
$u_0$ is uniformly continuous and satisfies a coercivity condition 
\begin{equation}\label{as-initial}
\lim_{R \to \infty} \inf\{u_0(y): |y| \ge R\} > u_0(x) \quad \text{for any} \; \; x \in \mathbb{R}^n. 
\end{equation}
\end{itemize}
The coercivity assumption \eqref{as-initial}  implies boundedness of all sublevel sets of $u_0$.

Equation \eqref{E1} can be derived by using the level set approach to the evolution of crystal growth type. More precisely, 
let $u(x,t)$ denote the height of a crystal from the flat plate at position $x$ and time $t$; consult \cite{G2} for more physical background. 
Assume that the crystal grows only horizontally, and that the normal velocity $V$ of the each $h$-level set of $u$ depends on its level and a nonlocal 
quantity of the enclosed region of its level, that is, 
\begin{equation}\label{eq:surface}
V=f(h, \mu(\Omega_t))
\quad \text{on} \ \{u(\cdot, t)=h\}, 
\end{equation}
where $\Omega_t:=\{y\in\R^n\mid u(y,t)<u(x,t)\}$. 
Such equations also appear in the singular limit of some nonlocal reaction diffusion equations \cite{CHL}.

Let us mention a few more related works on the study of level-set equations with nonlocal terms involving sublevel or superlevel sets of the unknown. In \cite{C}, Cardaliaguet established existence and uniqueness of solutions to a general class of nonlocal monotone geometric evolution equations; refer to \cite{CaPa} for related numerical analysis and \cite{Pas} for important applications in image processing. Later, the well-posedness and approximation schemes for the associated Neumann boundary problems are studied in \cite{Sl, DKS}. We refer to \cite{Sr} for existence and uniqueness results on a different type of second order nonlocal geometric equations  motivated by applications in tomographic image reconstruction. We also refer to a more recent work \cite{CMP} on the so-called nonlocal curvature flow, which is  similar to \eqref{eq:surface} but with $f$ depending on the space variable $x$ instead of the value of the unknown. Under the monotonicity condition, the classical viscosity solution theory including the comparison principle can be readily extended to this nonlocal problem. 
On the other hand, in the non-monotone case, one cannot expect the comparison principle to hold and alternative methods are needed to prove uniqueness of solutions and other related properties; see for instance \cite{ACM, BL} in the first order case and  \cite{BLM, KK} for second order problems. 

While the aforementioned results focus mostly on the well-posedness issue of nonlocal evolution equations, in this work we aim to further investigate the behavior of solutions in the case of first order problem \eqref{E1}\eqref{initial}. Thanks to the monotonicity of the nonlocal operator, we can obtain the uniqueness of viscosity solutions through a comparison principle, Theorem \ref{thm:comparison1}, whose proof is included in Appendix \ref{sec:app} for the reader's convenience. 
Hereafter, since we are always concerned with viscosity solutions of \eqref{E1}, \eqref{initial} defined in Definition \ref{def env-sol}, the term ``viscosity" is omitted henceforth. 

Our first main result is a representation formula of the unique solution to \eqref{E1}, \eqref{initial} from the viewpoint of optimal control theory. The strong connection between Hamilton-Jacobi equations (without nonlocal terms) and optimal control theory or differential games is well understood, as elaborated in \cite{LiBook, ESo, BCBook, TraBook}. We here introduce a nonlocal counterpart of these results. Below let us explain the heuristic idea to find our representation formula. Formally speaking, for any fixed $h\in \R$, 
one can rewrite \eqref{E1} as 
\begin{numcases}{}
u_t+|\nabla u| f(h, q_h(t))=0 &\  \text{in $\R^n\times (0, \infty)$,}\label{eq local1}\\ 
q_h(t)=\mu(\{u(\cdot, t)<h\})\label{eq local2}.
\end{numcases}
Suppose that $f(h, q_h(t))>0$ for all $t\geq 0$. Then letting $u(x, t)=v(x, \delta_h(t))$ with $\delta_h: [0, \infty)$ satisfying $\delta_h'(t)=f(h, q_h(t))$ for $t>0$, we can further write \eqref{eq local1} as 
\[
v_t+|\nabla v|=0\quad\text{in $\R^n\times (0, \infty)$}
\] 
without changing the initial data $u_0$. Note that for this simpler equation without nonlocal terms,  the classical optimal control theory \cite{BCBook} provides the following formula for $v$:
\[
v(x, t)=\inf_{|y-x|\leq t} u_0(y), \quad (x, t)\in \R^n\times [0, \infty). 
\]
Setting $D_h(0):=\{u_0<h\}$, we thus get, for any $t>0$, 
\[
\{v(\cdot, t)<h\}=\{x\in \R^n: \dist(x, D_h(0))<t\}, 
\]
where the right hand side, denoted also by $D_h(t)$, is known as the parallel set of $D_h(0)$ at distance $t$. The relation between $u$ and $v$ thus yields 
\begin{equation}\label{intro parallel}
\{u(\cdot, t)<h\}=\{v(\cdot, \delta_h(t))<h\}=D_h(\delta_h(t)).
\end{equation}
Hence, by \eqref{eq local2}, $\delta_h$ satisfies 
\begin{equation}\label{intro ode}
\delta_h'(t)=f(h, \mu(D_h(\delta_h(t)))).
\end{equation}
Once the ordinary differential equation \eqref{intro ode} is solved for all $h\in \R$ and $t\geq 0$, we can find the sublevel set $\{u(\cdot, t)<h\}$ via \eqref{intro parallel}, which further enables us to recover the representation formula of $u$. 

When carrying out our strategy above rigorously, we need to tackle two additional technical issues. In general, the evolutions of sublevel sets $\{u(\cdot, t)<h\}$ and $\{u(\cdot, t)\leq h\}$ may be different, which requires us to treat them separately. Besides, we also allow $f$ to change sign; in fact, this will be another assumption of this work (as in \eqref{h-threshold} below), which causes the nonlocal Hamiltonian to be nonconvex and noncoercive with respect to $\nabla u$, but also gives rise to more intricate behavior of solutions. One can follow the same argument to obtain similar results in the simpler cases when $f>0$ or $f<0$ in $\R\times [0, \infty)$. 

Let us now give a more precise description for our representation formula for solutions to \eqref{E1}, \eqref{initial}. For any $h\in\R$, we take 
\[
D_h(0):=\{u_0<h\}, \quad E_h(0):=\{u_0\leq h\}, 
\]
and define the parallel sets $D_h(s)$ and $E_h(s)$ for all $s\in\R$ of $D_h(0)$ and $E_h(0)$, respectively, by 
\begin{equation}\label{parallel set}
\begin{aligned}
&D_h(s) := \{x\in \R^n: \sd(x, D_h(0))<s\}, \\
&E_h(s) := \{x\in \R^n: \sd(x, E_h(0))\leq s\},
\end{aligned}
\end{equation}
where $\sd(x, A)$ denotes the signed distance to a set $A\subset \R^n$, namely,
\[
\sd(x, A)= {\rm dist}(x, A)- {\rm dist}(x, \R^n\setminus A).
\]
Note that, thanks to \eqref{as-initial}, $D_h(s)$ and $E_h(s)$ are bounded for all $h\in\R$ and $s\in\R$. 

Since we often deal with $D_h(s)$ and $E_h(s)$ in an analogous way, in the sequel we use $W_h(s)$ to represent $D_h(s), E_h(s)$ at the same time. 
Moreover, we denote by $\delta_W (t,h)$ the distance change of parallel sets $W_h(s)$ with respect to time $t$. As clarified in our formal derivation above,  
we expect that $\delta_W$ satisfies the ordinary differential equation 
\begin{equation}\label{dist ODE}
(\delta_W)_t(t,h) = (f_h\circ \mu)\left(W_{h, t}\right) \quad \text{for} \; \; t>0 \quad \text{with} \; \; \delta_W(0,h)=0, 
\end{equation}
where we set 
\begin{equation}\label{abbrev}
f_h := f(h, \cdot), \quad W_{h, t}:=W_h(\delta_W(t,h)).
\end{equation}
Note that $\mu(\inter A)$ and $\mu(\cl A)$ for a set $A \subset \R^n$ may not coincide in general, where $\inter A$ and $\cl A$ respectively denote the interior and the closure of $A$. We thus study a more general ordinary differential inequality 
\begin{equation}\label{dist evolution}
(f_h\circ\mu)(\inter W_{h, t}) \le (\delta_W)_t(t , h)\le (f_h\circ\mu)(\cl {W_{h, t}}) \quad \text{for a.e.} \;  t \ge 0 \ \ \text{with} \; \delta_W(0,h) = 0. 
\end{equation}
We give an existence results for \eqref{dist evolution} in Proposition \ref{prop:unique-ODE}. 
Since the uniqueness of solutions to \eqref{dist evolution} may not hold, we consider the maximal solution $\overline{\delta}_D(\cdot,h)$ to the interior evolution and the minimal solution $\underline{\delta}_E(\cdot,h)$ to the closure evolution, which are uniquely determined; see Proposition \ref{prop:sol-threshold} for more details.  For simplicity of our notation, we still use $\delta_W$ to denote such minimal and maximal solutions, namely, 
\begin{equation}\label{def-sol} 
\delta_E(t, h) := \overline{\delta}_E(t,h), \quad \delta_D(t,h) := \underline{\delta}_D(t,h) \quad \text{for} \; \; t \ge 0, \; h \in \R. 
\end{equation}

Suppose that there exists $\overline{h}\in\R$ depending on $u_0$ such that 
\begin{equation}\label{h-threshold} 
(f_{\overline{h}}\circ\mu)(D_{\overline{h}}(0)) \le 0 \le (f_{\overline{h}}\circ\mu)(E_{\overline{h}}(0)).  
\end{equation} 
Since $h\mapsto (f_h\circ \mu)({W_{h}(0)})$ is strictly increasing, it is easily seen that such constant is unique. 
Since the convexity of the Hamiltonian in gradient changes at $f=0$, our control setting also takes different forms depending on the sign of $h-\ol{h}$.
Set  
\[
U_0(x, h):=u_0(x)-h\quad \text{for $x\in \R^n$ and $h\in \R$}, 
\]
and define the value functions by 
\begin{align}\label{def:value}
&U_D(x, t, h):=
\left\{
\begin{array}{ll}
\min\{U_0(y, h): |y-x|\leq {\delta}_D(t, h)\} & \text{if} \ h > \overline{h}\\
\max\{U_0(y, h): |y-x|\leq -{\delta}_D(t, h)\} & \text{if} \ h \le \overline{h}, 
\end{array}
\right. \\
&U_E(x, t, h):=
\left\{
\begin{array}{ll}
\min\{U_0(y, h): |y-x|\leq {\delta}_E(t, h)\} & \text{if} \ h \ge \overline{h}\\
\max\{U_0(y, h): |y-x|\leq -{\delta}_E(t, h)\} & \text{if} \ h < \overline{h}.  
\end{array}
\right. \label{def:value2}
\end{align}

Since we are interested in the solution to \eqref{E1} and \eqref{initial}, for each $t\geq 0$, we take back the zero level set of $U(x, t, h)$ to recover the evolution starting from $u_0$. To this end, we set, for each $(x, t) \in\R^n\times [0, \infty)$,
\begin{equation}\label{new upper value}
u_D(x, t)=\sup\{h\in \R: U_D(x, t, h)\geq 0\},
\end{equation}
\begin{equation}\label{new lower value}
u_E(x, t)=\inf\{h \in \R: U_E(x, t, h)\leq 0\}.  
\end{equation}
It turns out that, for all $x\in \R^n$ and $t\geq 0$, we have
\begin{equation}\label{bounds}
\min\{u_0(x),\ \ol{h}\} \leq u_E(x, t) \leq u_D(x, t)  \leq \max\{u_0(x),\ \ol{h}\} \quad \text{for all $(x, t)\in\R^n\times [0, \infty)$}; 
\end{equation}
see Proposition \ref{prop direct compare} for the proof. 
Our definition of $u_D$ and $u_E$ using the height functions $U_D$ and $U_E$ adapts the original idea in \cite{KS2} to our nonlocal problem. 


%

\begin{thm}\label{thm:control} 
Assume that (A1)--(A4) hold and there exists $\ol{h}\in \R$ satisfying \eqref{h-threshold}. Let $u_D$ and $u_E$ be the functions defined, respectively, by \eqref{new upper value} and \eqref{new lower value}.  Then $u:=u_D=u_E$ is the unique continuous solution of \eqref{E1}, \eqref{initial} satisfying the growth condition that,  
for any $T>0$, there exists $M_T > 0$ such that 
\begin{equation}\label{eq:growth} 
|u(x,t)| \le M_T (|x| + 1) \quad \text{for all} \; \; (x,t) \in \R^n \times [0,T]. 
\end{equation}
\end{thm}
As an application of Theorem \ref{thm:control}, we study the fattening phenomenon for 
$h$-level sets for all $h\in\R$. We say that the fattening phenomenon takes place for the solution $u$ of a level set evolution equation at level $h$ if
\begin{equation}\label{initial regular}
\cl\{u(\cdot, t)< h\}=\{u(\cdot, t)\leq h\} \quad \text{and}\quad \{u(\cdot, t)<h\}=\inter\{u(\cdot, t)\leq h\}
\end{equation}
holds for $t=0$ but fails to hold at some $t>0$. 
The level flow is also called regular if \eqref{initial regular} holds for all $t\geq 0$. The occurrence of fattening is first found by Evans and Spruck \cite{ES1} for level set mean curvature flow equation. Consult \cite{BSS, So3, BP, G1, GBook} for other fattening and non-fattening results for various geometric evolution problems without nonlocal terms.

We extend the study of fattening to the nonlocal equation \eqref{E1} by applying our control-theoretic interpretation in Theorem \ref{thm:control}. The following result is about the case when $h\neq \ol{h}$. 
\begin{thm}\label{thm:non-th-fatten} Assume that (A1)--(A4) hold and there exists $\ol{h}\in \R$ satisfying \eqref{h-threshold}. 
Let $u$ be the unique solution to \eqref{E1}, \eqref{initial} satisfying \eqref{eq:growth}.  
Fix $h\in\R\setminus\{\overline{h}\}$ arbitrarily.   
\begin{enumerate}
\item[{\rm(a)}]
If $D_h(0)$ and $E_h(0)$ satisfy 
\begin{equation}\label{cond:nonfat}
 \cl {D}_h(0)=E_h(0) \ \text{when} \ h>\overline{h}, 
\ \
D_h(0)=\inter E_h(0) \ \text{when} \ h<\overline{h},  
\end{equation}
then $\delta_D(t, h) = \delta_E(t,h)$ holds for all $t>0$, 
that is, the fattening phenomenon never happens at $h$-level set, and, for any $t \ge 0$,  
\[ \{x \in \mathbb{R}^n: u(x,t)=h\} = 
\begin{cases}
\partial D_h(\delta_D(t,h)) & \text{if} \; \; h > \overline{h}, \\
\partial E_h(\delta_E(t,h)) & \text{if} \; \; h < \overline{h}.   
\end{cases} \]

\item[{\rm(b)}] 
If $D_h(0)$ and $E_h(0)$ satisfy
\begin{equation}\label{as:fattening}
 \cl {D}_h(0) \subsetneq E_h(0) \ \text{when} \ h>\overline{h}, 
\ \ 
D_h(0) \subsetneq \inter E_h(0) \ \text{when} \ h<\overline{h}, 
\end{equation}
then $\delta_E(t, h) > \delta_D(t, h)$ holds for all $t > 0$. 
\end{enumerate}
\end{thm}
This result can be viewed as a nonlocal adaptation of the non-fattening criterion established by \cite[Theorem 4.1]{BSS} for first order geometric evolutions where the normal velocity does not change sign. However, we remark that for our nonlocal equation the fattening phenomenon may take place in a subset of the space. An example for Theorem \ref{thm:non-th-fatten}(b) is the evolution from $u_0\in C(\R)$ whose $h$-sublevel set $\{x: u_0(x) \le h\}$ consists of a closed interval and an isolated point, as illustrated in Figure \ref{fig:1.1} for the case $h>\ol{h}$. Due to the nonlocal nature of the equation, its solution develops an interior instantaneously not only at the isolated point but also near the boundary of the interval. A more specific example for this phenomenon is presented in Example~\ref{ex partial fattening}.  

\begin{figure}[H]
  \centering
  \includegraphics[height=3.8cm]{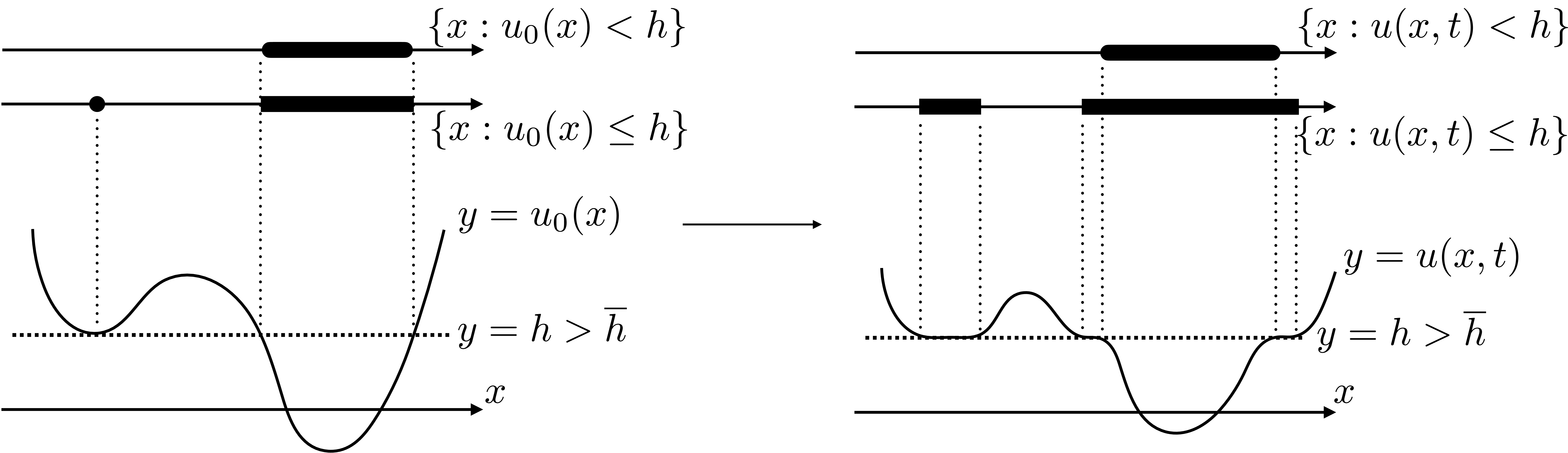}
  \caption{The figure illustrates the behavior of $h$-level set of a solution $u$ when $h > \overline{h}$ and $\{x: u_0(x) \le h\}$ comprises an isolated point and a closed interval.
  }
  \label{fig:1.1}
 \end{figure}



The situation for the critical level $h=\ol{h}$ is more involved. Even if the $\ol{h}$-level set of the solution $u$ satisfies \eqref{initial regular} at $t=0$ and has zero initial normal velocity in the sense that
\[
f(\ol{h}, \mu\{u(\cdot, 0)<\ol{h}\})=f(\ol{h}, \mu\{u(\cdot, 0)\leq \ol{h}\})=0,
\]
it may still have nonempty interior for $t>0$. We construct a fattening example in Example \ref{ex critical fattening} by choosing the initial data to be the set provided in \cite[Example 1]{Kr}. For general discussions on the fattening problem at the critical level $\overline{h}$, it turns out that the rate of the measure growth for sublevel sets of $u_0$ at neighboring levels plays a key role.
Several related fattening and non-fattening conditions are provided in Theorems \ref{prop:fattening-o-h-1}--\ref{prop:fattening-o-h-3}. 


 We finally study the large-time behavior for solutions to \eqref{E1}\eqref{initial}. Compared to the large time asymptotics for Hamilton-Jacobi equations without nonlocal terms, for which we refer to \cite{LMT} and references therein, the case of nonlocal Hamilton-Jacobi equations is less well-studied. For our nonlocal problem \eqref{E1} \eqref{initial}, since we can easily see that 
 $u_t\le 0$ if $u<\overline{h}$, and  $u_t\ge 0$ if $u>\overline{h}$, 
 we expect locally uniform convergence of $u(\cdot, t)$ to $\ol{h}$, that is, 
  \begin{equation}\label{eq:conv}
\sup_{K} |u(\cdot, t)- \overline{h}|\to 0 \quad \text{as} \ t\to\infty \ \text{for any compact set} \ K\subset \R^n, 
 \end{equation}
 which is indeed proved in Proposition \ref{prop:conv} under (A1)--(A4). 
 Our interest is to further obtain finer results on the asymptotic behavior. 
In fact, depending on the initial condition and the location of $x\in\R^n$, 
one may find the following two types of behavior: 
\begin{align*}
& \text{Type I (Finite-time stabilization):} \ \ \text{there exists} \ T>0 \ \text{such that} \ u(x,t)=\overline{h} \ \text{for all} \ t\ge T, \\
&\text{Type II (Infinite-time convergence):} \ \ u(x, t)\to \ol{h} \text{ as $t\to \infty$ but } u(x,t)\neq \ol{h} \ \text{for all} \ t>0.
\end{align*}
 In Proposition \ref{cor:finite-infinite} we prove that the behavior in Type I holds for any $x\in \R^n$ satisfying $u_0(x)<\overline{h}$ if $(f_{\overline{h}}\circ\mu)({D}_{\overline{h}}(0))<0$ and for any $x\in \R^n$ satisfying $u_0(x)>\overline{h}$ if $(f_{\overline{h}}\circ\mu)({E}_{\overline{h}}(0))>0$.
On the other hand, under the assumption that $f(\overline{h}, \cdot)$ is of class ${\rm Lip}_{\rm loc}([0, \infty))$, we show that Type II holds for each $x$ satisfying $u_0(x) < \overline{h}$, Type II holds if 
\begin{equation}\label{type2d}
(f_{\overline{h}} \circ \mu)(D_{\overline{h}}(0)) = 0 \quad \text{and} \quad {\rm Per}(D_{\overline{h}}(-s)) \le C \; \; \text{for} \; \; 0 < s \le s_0 \quad \text{for some $C, s_0 > 0$,}
\end{equation}
 where we denote by ${\rm Per}(A)$ the perimeter of $A\subset\R^n$.
Similarly, for each $x$ satisfying $u_0(x) > \overline{h}$, we have the convergence of Type if 
\begin{equation}\label{type2e} 
(f_{\overline{h}} \circ \mu)(E_{\overline{h}}(0)) = 0 \quad \text{and} \quad {\rm Per}(E_{\overline{h}}(s)) \le C \; \; \text{for} \; \; 0 < s \le s_0\quad \text{for some $C, s_0 > 0$.} 
\end{equation}

 
In addition, we can further establish a convergence rate of $u(\cdot, t)$ to $\overline{h}$ as $t\to\infty$, regardless of the asymptotic types. 

\begin{thm}\label{thm exponent}
Assume that (A1)--(A4) hold and there exists $\ol{h}\in \R$ satisfying \eqref{h-threshold}. Let $u$ be the unique solution to \eqref{E1} and \eqref{initial} satisfying \eqref{eq:growth}. 
Assume in addition that 
 the following two conditions hold: 
\begin{equation}\label{reg initial lower}
 \limsup_{s\to 0+}\frac{f(\ol{h}, \mu(D_{\ol{h}}(-s)))}{s} <0; \hspace{2cm} 
\end{equation}
\begin{equation}\label{reg initial}
\left\{\ 
\begin{aligned}
&\text{For any $b>0$ small,  there exists $a>0$  such that }\\
 & f(h, \mu(E_h(s)))-f(h, \mu(E_h(0)))\geq a s  \quad \text{for all $\ol{h}\leq h\leq \ol{h}+b$, $0\leq s\leq 1/b$.}
  \end{aligned}
  \right.
\end{equation}
Then for any compact set $K\subset \R^n$, there exist $C_K>0$ and $\lambda>0$ such that 
\begin{equation}\label{large time est}
|u(x, t)-\ol{h}|\leq g(C_K e^{-\lambda t})\quad \text{for all $x\in K$ and $t\geq 0$ large,}
\end{equation}
where $g:[0, \infty)\to [0, \infty)$ denotes a modulus of continuity of $u_0$, i.e., $g$ is a continuous, strictly increasing function satisfying $g(0)=0$. 
In particular, if $u_0$ is H\"{o}lder continuous with component $\alpha\in (0, 1]$, then for any compact set $K\subset \R^n$,  there exists $C_K>0$ such that $|u(x, t)-\ol{h}|\leq C_K e^{-\alpha \lambda t}$ for all $x\in K$ and all $t>0$ large.
\end{thm}

Our proof of Theorem \ref{thm exponent} is still based on the optimal control interpretation in Theorem \ref{thm:control}. By constructing appropriate barrier functions, we give estimates for the lower bound and upper bound of $u(\cdot, t)-\ol{h}$ separately; 
more details can be found in Proposition \ref{prop exponent-lower} and Proposition \ref{prop exponent-upper}. Under the assumptions (A1)--(A4) and the non-degeneracy condition \eqref{reg initial lower} associated to the initial critical level set, the lower bound can be justified also by an alternative PDE method; in fact, the lower barrier function we construct is indeed a subsolution of \eqref{E1}, which enables us to use the comparison principle to directly deduce the estimate. 

For the upper bound of $u(\cdot, t)-\ol{h}$, one may expect that a similar estimate holds under a symmetric counterpart of \eqref{reg initial lower} such as
\begin{equation}\label{reg initial upper}
 \limsup_{s\to 0+}\frac{f(\ol{h}, \mu(E_{\ol{h}}(s)))}{s} >0,
\end{equation}
but the situation is actually different, due to additional difficulty arising from the unboundedness of $\{u(\cdot, t)\geq \ol{h}\}$. The same PDE method does not seem to work well. We obtain the estimate still by the control-theoretic approach under \eqref{reg initial}, which requires $f$ to satisfy a more restrictive non-degeneracy condition than \eqref{reg initial upper}.  A more specific sufficient condition for \eqref{reg initial} involving quasiconvexity of $u_0$ is given in Proposition \ref{prop reg initial}. 

It is worth remarking that in some special cases the exponent $\lambda>0$ in \eqref{large time est} may be taken independent of the choice of $K$. This improved result is discussed in Remark \ref{rmk improved decay} and a related simple example is presented in Example \ref{ex decay}.

The paper is organized as follows. 
At the beginning of Section \ref{sec:rep}, we recall the definition of viscosity  solutions and the comparison principle for \eqref{E1}. The rest of Section \ref{sec:rep} is devoted to the proof of Theorem \ref{thm:control}. In Section \ref{sec:fat}, we apply Theorem \ref{thm:control} to study the fattening/non-fattening phenomenon, including a proof of Theorem \ref{thm:non-th-fatten}. 
Finally, in Section \ref{sec:large} we investigate the large time asymptotics of solutions to \eqref{E1} \eqref{initial}.

\subsection*{Acknowledgments}
We would like to thank Prof. Keisuke Takasao for bringing the reference \cite{Kr} to our attention. 
TK was partially supported by Japan Society for the Promotion of Science (JSPS) through grants: KAKENHI \#20H01801, \#21H00990, \#23K12992, \#23H00085. 
QL was partially supported by the JSPS grants KAKENHI \#19K03574, \#22K03396.
HM was partially supported by the JSPS grants: KAKENHI
\#22K03382, \#21H04431, \#20H01816, \#19H00639. 

\section{Control-based representation formulas}\label{sec:rep}

\subsection{Preliminaries}\label{sec:pre} 
We first recall the definition of viscosity solutions to \eqref{E1}. Consult \cite{Sl} for the definition of viscosity solutions to more general equations.
\begin{defn}\label{def env-sol}
A locally bounded upper semicontinuous (USC) function $u: \R^n\times (0, \infty)\to \R$ is called a viscosity subsolution of \eqref{E1} if whenever there exist $(x_0, t_0)\in \R^n\times (0, \infty)$ and $\varphi\in C^1(\R^n\times (0, \infty))$ such that $u-\varphi$ attains a local maximum at $(x_0, t_0)$, 
\begin{equation}\label{sub eq}
\varphi_t(x_0, t_0)+|\nabla \varphi(x_0, t_0)|f\big(u(x_0, t_0), \mu(\{u(\cdot, t_0)< u(x_0, t_0)\})\big)\leq 0.
\end{equation}

Symmetrically, a locally bounded lower semicontinuous (LSC) function $u: \R^n\times (0, \infty)\to \R$ is called a viscosity supersolution of \eqref{E1} if whenever there exist $(x_0, t_0)\in \R^n\times (0, \infty)$ and $\varphi\in C^1(\R^n\times (0, \infty))$ such that $u-\varphi$ attains a minimum at $(x_0, t_0)$, 
\[
\varphi_t(x_0, t_0)+|\nabla \varphi(x_0, t_0)|f\big(u(x_0, t_0), \mu(\{u(\cdot, t_0)\leq u(x_0, t_0)\})\big)\geq 0.
\]
A function $u\in C(\R^n\times (0, \infty))$ is called a viscosity solution of \eqref{E1} if it is both a viscosity subsolution and a viscosity supersolution. 
\end{defn}

\begin{thm}[Comparison principle]\label{thm:comparison1}
Let $u \in USC(\R^n \times [0,\infty))$ and $v \in LSC(\R^n \times [0,\infty))$ be, respectively, a subsolution and a supersolution to \eqref{E1}. 
Assume in addition that for any $T>0$, there exists $M_T > 0$ such that 
\[ u(x,t) \le M_T (|x| + 1), \quad v(x,t) \ge - M_T(|x| + 1) \quad \text{for} \; \; (x,t) \in \R^n \times [0,T]. \]
If there exists a modulus of continuity $\omega_0$ such that 
\begin{equation}\label{initial_modulus} 
u(x,0) - v(y,0) \le \omega_0(|x-y|) \quad \text{for} \; \; x,y \in \R^n, 
\end{equation}
then $u \le v$ holds in $\R^n \times [0,\infty)$. 
\end{thm}
See Appendix \ref{sec:app} for the proof. A comparison result for more general first order nonlocal equations can be found in \cite{KLM-rims}.

\subsection{Maximal and minimal solutions of parallel set evolutions}

Let us make preparations for our proof of Theorem \ref{thm:control}. 
We recall that $D_h(s)$ and $E_h(s)$ are the parallel sets of $D_h(0), E_h(0)$ defined by \eqref{parallel set}. 
It is clear that $D_h(s)$ is open, $E_h(s)$ is closed, and $D_h(s)\subset E_h(s)$ for all $s\in \mathbb{R}$ and  $h\in \R$. 
Also, we note that in view of (A2) we have the monotonicity 
\begin{equation}\label{eq:mono}
\cl{D_{h_1}}(s)\subset D_{h_2}(s), 
\quad\text{and} \quad 
E_{h_1}(s)\subset \inter E_{h_2}(s) \quad
\text{if} \ h_1<h_2. 
\end{equation}

Let us first mention the result \cite[Theorem 3]{Kr}, which states that there exists a constant $C>0$ such that
\begin{equation}\label{bdd-boundary} 
{\rm Per}(\{x \in \mathbb{R}^n: {\rm dist} \; (x, A) < s\}) \le \frac{C}{s} m(\{x \in \mathbb{R}^n: 0 < {\rm dist} (x, A) < s\}) 
\end{equation}
holds for any $s>0$ and any subset $A \subset \mathbb{R}^n$. Here ${\rm Per}(A)$ denotes the perimeter of $A$. This can be viewed as a certain inverse isoperimetric inequality. Applying (A3), we then have 
\begin{equation}\label{same-cl-in} 
\mu(\cl({W_h(s)})) = \mu(\inter(W_h(s))) \quad \text{for} \; h\in \R, \; s\neq 0, \; W=D, E.  
\end{equation}
We use this relation to show that any solution of  \eqref{dist evolution} actually solves \eqref{dist ODE} for $h \neq \overline{h}$. 

\begin{lem}\label{lem:reform-ODE}
Assume that (A1)--(A4) hold and there exists $\ol{h}\in \R$ satisfying \eqref{h-threshold}.
For any $h\in\R\setminus\{\overline{h}\}$, let $\delta_W(\cdot, h)\in\Lip([0,\infty))$ be any function satisfying 
\eqref{dist evolution} for $W=D$ or $W=E$.  
Then, $\delta_W(\cdot, h)$ is of $C^1((0,\infty))$ and solves \eqref{dist ODE}.
Moreover, $\delta_W(\cdot, h)$ is strictly increasing if $h> \overline{h}$, and 
$\delta_W(\cdot,h)$ is strictly decreasing if $h<\overline{h}$. 
\end{lem}

%

\begin{proof}
We only consider the case $h>\overline{h}$ as we can prove similarly in the case $h<\overline{h}$. We adopt the notations $f_h$ and $W_{h, t}$ as given in \eqref{abbrev}.
By the monotonicity \eqref{eq:mono},  
we have $\inter E_h(0)\supset E_{\overline{h}}(0)$. 
Thus, by the continuity of $\delta_W(t, h)$ in $t \ge 0$, there exists $\ep > 0$ such that 
\[ 
\inter W_{h, t}=\inter E_h(\delta_{W}(t,h)) \supset E_{\overline{h}}(0) \quad \text{for} \; \; 0 \le t < \ep. 
\]
Since $f_h$ is strictly increasing, we get
\[ 
0\le 
(f_h \circ \mu)(E_{\overline{h}}(0))< (f_h \circ \mu)(\inter W_{h,t}) \leq (\delta_W)_t(t,h), 
\]
which implies that 
$\delta_W(\cdot,h)$ is strictly increasing. 

Noticing that $\delta_W(t, h) > 0$ for any $t>0$, by \eqref{same-cl-in}, we see that $\delta_W(\cdot,h)$ solves \eqref{dist ODE} for $t>0$. 
In light of the continuity of the measure, we have, for $t>0$,
\begin{equation}\label{cont-wrt-set}
\begin{aligned} 
& \ \lim_{\varepsilon \to 0+}  (f_h \circ \mu)(W_{h, t-\ep}) =\; (f_h \circ \mu)\left(\bigcup_{\varepsilon > 0} W_h(\delta_W(t-\varepsilon), h))\right) = (f_h \circ \mu)(\inter W_{h,t}) \\
&\qquad \  = (f_h \circ \mu)(\cl{W_{h,t}}) = (f_h \circ \mu)\left(\bigcap_{\varepsilon > 0} W_h(\delta_W(t +\varepsilon, h))\right)  = \lim_{\varepsilon \to 0+} (f_h \circ \mu)(W_{h, t+\ep})  
\end{aligned}
\end{equation}
 owing to the monotonicity of $\delta_W(\cdot, h)$ and \eqref{same-cl-in},  
which implies the continuity of the right hand side of \eqref{dist ODE}, and thus $\delta_W(\cdot, h)$ is of $C^1((0,\infty))$. 
\end{proof}

We next study the existence of solutions to \eqref{dist ODE}.

\begin{prop} \label{prop:unique-ODE} 
Assume that (A1)--(A4) hold. Then there exists a solution $\delta_W(\cdot,h) \in C^1([0,\infty))$ to \eqref{dist ODE} for 
any $h\in\R$, $W=D, E$. 
\end{prop}

\begin{proof}
Let $\overline{h} \in \mathbb{R}$ be the constant satisfying \eqref{h-threshold}. 
We first consider the case where $h > \overline{h}$ with $W=D$ or $W=E$. 
Define the function $\tilde{f}: \R \to \R$ by 
\[ 
\tilde{f}(s) := 
\begin{cases}
(f_h \circ \mu) (W_h(s)) & \text{if} \; \; s > 0, \\
(f_h \circ \mu) (\cl{W_h(0)}) & \text{if} \; \; s \le 0. 
\end{cases} 
\]
By the same argument as in \eqref{cont-wrt-set} 
we obtain the continuity of $\tilde{f}$. 
Therefore, in light of the Peano existence theorem for ordinary differential equations, 
and the boundedness of $\tilde{f}$, 
there exists a solution $\delta_W(\cdot, h) \in C^1([0, \infty))$ to 
\[ 
(\delta_W)_t (t, h) = \tilde{f}(\delta_W(t,h)) \quad \text{for} \; \; t > 0 \quad \text{with} \; \; \delta_W(0, h) = 0. 
\]
Due to \eqref{h-threshold}, it is obvious that $\delta_W(\cdot, h)$ is strictly increasing, and thus $\delta_W(\cdot,h)$ is a solution to \eqref{dist ODE} by the definition of $\tilde{f}$.  The same proof above also works for the case when $h=\ol{h}$ and $W=E$. 

By symmetry, we can prove the existence of solutions to \eqref{dist ODE} 
in the remaining cases that $h < \overline{h}$ with $W=D, E$ as well as $h=\overline{h}$ with $W=D$.
\end{proof}

It is clear that any solution to \eqref{dist ODE} satisfies \eqref{dist evolution}. 
However, in the case $h=\overline{h}$, a function satisfying \eqref{dist evolution} may not be a solution to \eqref{dist ODE}, and solutions to \eqref{dist evolution} may not be unique in general. 
Therefore, we introduce the maximal and minimal solutions for \eqref{dist evolution}.

\begin{prop} [Maximal and minimal solutions to \eqref{dist evolution}] \label{prop:sol-threshold} 
Assume that (A1)--(A4) hold and there exists $\ol{h}\in \R$ satisfying \eqref{h-threshold}. There exist a unique Lipschitz  maximal solution $\ol{\delta}_W$ and a unique Lipschitz minimal solution $\underline{\delta}_W(\cdot, h)$ to \eqref{dist evolution}, that is, %
\[
\underline{\delta}_W(t, h) \le \delta_W(t, h) \le  \overline{\delta}_W(t, h)
\]
for all $t \ge 0$ and any solution $\delta_W(\cdot, h)$ to \eqref{dist evolution}. 
Moreover, in the case $h= \overline{h}$, where $\overline{h}$ is the constant satisfying \eqref{h-threshold}, $\underline{\delta}_D(\cdot, \overline{h})$ is nonincreasing and $\overline{\delta}_E(\cdot, \overline{h})$ is nondecreasing. 
\end{prop}

\begin{proof}

By the standard ODE theory, there exists a Lipschitz solution of \eqref{dist evolution}. Note that the Lipschitz constant can be taken to be $M:=\sup_{\R\times [0, \infty)} |f|$, thanks to (A1). We define $\overline{\delta}_W(\cdot, h)$ and $\underline{\delta}_W(\cdot, h)$ by 
\begin{align*} 
&\overline{\delta}_W(t, h) := \sup\{ \delta_W(t, h): \text{$\delta_W(\cdot, h)$ is an $M$-Lipschitz solution of \eqref{dist evolution}} \}, \\
&\underline{\delta}_W(t, h) := \inf\{ \delta_W(t, h): \text{$\delta_W(\cdot, h)$ is an $M$-Lipschitz solution of \eqref{dist evolution}} \}. 
\end{align*}
It is clear that $\overline{\delta}_W(t,h), \underline{\delta}_W(t,h)$ are both $M$-Lipschitz in $[0, \infty)$. 
It then suffices to show that $\overline{\delta}_W(t,h), \underline{\delta}_W(t,h)$ solve \eqref{dist evolution}.

For any $t>s\geq 0$ with $\vep=t-s>0$ small, by the definition of $\underline{\delta}_W(\cdot, h)$, we can find a solution $\delta_W^\ast(\cdot, h)$ such that $\ul{\delta}_W(s, h)\leq \delta_W^\ast(s, h)\leq \ul{\delta}_W(s, h)+\vep^2$. It follows that 
\[
\ul{\delta}_W(t, h)-\ul{\delta}_W(s, h)\leq \delta_W^\ast(t, h)-\delta_W^\ast(s, h)+\vep^2\leq \int_s^t (f_h\circ \mu)(\cl W_{h,\tau}^\ast)\, d\tau+\vep^2, 
\]
where we set $W_{h,\tau}^\ast:= W_h(\delta_W^{\ast}(\tau, h))$ for $\tau\geq 0$. 
By the continuity of $\mu$ and $q\mapsto f(h, q)$ as well as the $M$-Lipschitz continuity of $\delta_W^\ast$, we have 
\begin{equation}\label{eq sol minmax1}
\ul{\delta}_W(t, h)-\ul{\delta}_W(s, h)\leq \vep  (f_h\circ \mu)(\cl W_{h, s}^\ast)+o(\vep). 
\end{equation}
Note also that 
\begin{equation}\label{eq sol minmax2}
\limsup_{t\to s+} \cl W^\ast _{h, s}=\bigcap_{\tau>0} \bigcup_{s<t<s+\tau} \cl W^\ast_{h, s}=\cl W_h(\ul{\delta}_W(s, h)).
\end{equation}
If $\ul{\delta}_W(\cdot, h)$ is differentiable at $s$, then dividing both sides of \eqref{eq sol minmax1} by $\vep$ and letting $t\to s$ with $\vep\to 0$, we obtain from \eqref{eq sol minmax2} and the upper semicontinuity of $\mu$ that
\[
(\ul{\delta}_W)_t(s, h)\leq (f_h\circ \mu)(\cl W_h(\ul{\delta}_W(s, h))).
\]
A symmetric argument yields 
\[
(\ul{\delta}_W)_t(s, h)\geq (f_h\circ \mu)(\inter W_h(\ul{\delta}_W(s, h)))
\]
for almost all $s\geq 0$.


By the definition of $\underline{\delta}_W(t,h)$, we can choose a sequence of solutions $\delta^{(i)}_W(\cdot, h)$ such that $\delta^{(i+1)}_W(\cdot, h) \le \delta^{(i)}_W(\cdot, h)$ for any $i \in \mathbb{N}$ and 
\[ \lim_{i \to \infty} \delta^{(i)}_W(t, h) = \underline{\delta}_W(t, h) \quad \text{for} \; \; t \ge 0. \]
Set $W_{h,t}^{i} := W_h(\delta_W^{(i)}(t, h))$ and $W^{\infty}_{h, t} := W_h(\underline{\delta}_W (t,h))$. 
We have 
\[ \inter W^{\infty}_{h,t} \subset \bigcap_{i=1}^\infty \inter W^{i}_{h,t} \subset \bigcap_{i=1}^\infty \cl{W^{i}_{h,t}} \subset \cl{W^\infty_{h,t}} \quad \text{for} \; \; t \ge 0. \]
Therefore, we deduce that for any $t \ge 0$,
\[ (f_h \circ \mu) (\inter W^{\infty}_{h,t}) \le \lim_{i \to \infty} (f_h \circ \mu)(\inter W_{h, t}^{i}), \quad (f_h \circ \mu)(\cl{W^{\infty}_{h, t}}) \ge \lim_{i \to \infty} (f_h \circ \mu) (\cl{W^{i}_{h, t}}). \]
Since  $\delta^{(i)}_W(\cdot, \overline{h})$ solves \eqref{dist evolution} for all $i$, taking the limit as $i\to \infty$ we see that $\underline{\delta}_W(t, h)$ is also a solution to \eqref{dist evolution}. One can similarly verify that $\overline{\delta}_W(t,h)$ is also a Lipschitz solution to \eqref{dist evolution}.

We finally prove that $\underline{\delta}_D(\cdot, \overline{h})$ is nonincreasing and $\overline{\delta}_E(\cdot, \overline{h})$ is nondecreasing. 
Since $\delta(t, \overline{h}) \equiv 0$ is a solution to \eqref{dist evolution} with $h=\overline{h}$, the maximal  and minimal solutions satisfy 
\[ \underline{\delta}_D(t, \overline{h}) \le 0 \le \overline{\delta}_E(t, \overline{h}) \quad \text{for} \; \; t \ge 0. \]
Therefore, due to the monotonicity of $f$ as in (A2), we can see that $\underline{\delta}_D(\cdot, \overline{h})$ is nonincreasing and $\overline{\delta}_E(\cdot, \overline{h})$ is nondecreasing. 
\end{proof}

Hereafter, we redefine $\delta_W(\cdot, h)$ for $h\in\R$, $W=D, E$ as in \eqref{def-sol}. 
Notice that we continue to use $\delta_D, \delta_E$ by abuse of notations. 

\begin{lem}[Monotonicity and continuity in $h$] \label{lem:ode} 
{Assume that (A1)--(A4) hold.} Let $h\in\R$ and $\delta_W(\cdot, h)$ be the function defined by \eqref{def-sol}.  
Then the following properties hold. 
\begin{enumerate}[{\rm(a)}]
\item 
$\delta_{D}(t, h_1) < \delta_{D}(t, h_2)$, $\delta_{D}(t, h_1) < \delta_{E}(t, h_2)$ and $\delta_{E}(t, h_1) < \delta_{E}(t, h_2)$ 
for any $t>0$ and $h_1 < h_2$. 
\item $h\mapsto \delta_D(t, h)$ is left continuous in $\R$, that is, for any $t>0$ and $h\in \R$, 
\begin{equation*}
\delta_D(t, h-\vep)\to \delta_D(t, h)\quad \text{as $\vep\to 0+$}. 
 \end{equation*}
 \item $h\mapsto \delta_E(t, h)$ is right continuous in $\R$, that is, for any $t>0$ and $h\in \R$, 
\begin{equation*}
\delta_E(t, h+\vep)\to \delta_E(t, h)\quad \text{as $\vep\to 0+$.}
 \end{equation*}
\item $\delta_D(t, h)\leq \delta_E(t, h)$ for all $t\geq 0$ and $h\in \R$. 
\end{enumerate}
\end{lem}

\begin{proof}
The statement (a) is a direct consequence of \eqref{eq:mono}, (A2), and \eqref{dist evolution}.
Let us prove (b). Note that for any $h\in \R$, there exists $\delta \in {\rm Lip}([0, \infty))$ such that 
$\lim_{\ep \to 0+} \delta_D(t, h-\ep) = \delta(t)$ for  $t \ge 0$. 
We then obtain 
\[ \bigcup_{\ep > 0} D_{h - \ep, t} = D_h(\delta(t)) \quad \text{for} \; \; t \ge 0. \]
Moreover, we have $h-\vep\neq \ol{h}$ for $\vep>0$ sufficiently small. By Lemma \ref{lem:reform-ODE}, we see that \eqref{dist ODE} holds with $h$ replaced by $h-\vep$. Passing to the limit 
as $\ep \to 0+$, we have 
\begin{equation}\label{ODE limit} 
\delta_t (t) = (f_h \circ \mu) (D_h(\delta(t))) \quad \text{for a.e.} \; \; t \ge 0 \quad \text{with} \; \; \delta(0) = 0. 
\end{equation}
In particular, $\delta(t)$ satisfies \eqref{dist evolution}. 
On the other hand, by (a), we see that 
\[ \delta(t) = \lim_{\ep \to 0+} \delta_D(t, h-\ep) \le \delta_D (t, h) \quad \text{for} \; \; t \ge 0. \]
Since $\delta_D(\cdot, h)$ was chosen as the minimal solution to \eqref{dist evolution} in \eqref{def-sol}, we have $\delta(t) \equiv \delta_D(t, h)$, 
which implies the left continuity of $\delta_D(t,h)$. 
The right continuity of $\delta_E$ can be proved symmetrically. 
Finally, since $\delta_{D}(t,h-\ep)<\delta_{E}(t,h+\ep)$ holds for all $\vep>0$,  letting $\ep\to0+$ yields (d) in light of (b) and (c).  
\end{proof}

Our argument in the proof of Lemma \ref{lem:ode}(b) above enables us to also show that $\delta_W$ given by \eqref{def-sol} are  solutions of \eqref{dist ODE} for $h=\ol{h}$. This is true in the case $h\neq \ol{h}$ as well due to Lemma \ref{lem:reform-ODE} and \ref{prop:sol-threshold}.

\begin{cor}\label{cor:ODE-oh} 
Assume that (A1)--(A4) hold and there exists $\ol{h}\in \R$ satisfying \eqref{h-threshold}. Let $h\in\R$ and $\delta_W(\cdot, h)$ be the function defined by \eqref{def-sol}.  
Then, $\delta_W(\cdot, \overline{h})$ is of class $C^1((0,\infty))$ and is a solution to \eqref{dist ODE}, 
respectively, for $W=D$ and $W=E$. 
\end{cor} 

\begin{proof}
When $h\neq \ol{h}$, the proof immediately follows from Lemma \ref{lem:reform-ODE} and \ref{prop:sol-threshold}. 
We only give a proof for $\delta_D(\cdot, \overline{h})$ since we can similarly prove for $\delta_E(\cdot, \overline{h})$. 
We have shown in the proof of Lemma \ref{lem:ode} that $\delta(t)=\delta_D(t, \ol{h})$ is a Lipschitz solution of \eqref{ODE limit} with $h=\ol{h}$.  It remain to prove $C^1$ continuity of $\delta_D(\cdot, \ol{h})$.
If $\delta_D(t, \overline{h}) = 0 $ for all $t\geq 0$, then $(f_{\overline{h}} \circ \mu)(D_{\overline{h}, t})$ is identically $0$ in time and, in particular, is continuous in $t \in [0, \infty)$.  
Suppose that there exists $t_0\ge0$ such that $\delta_D(t, \overline{h}) = 0$ for $t \in [0, t_0]$ and $\delta_D(\cdot, \overline{h})$ is strictly decreasing for $t \in (t_0, \infty)$. Since $\delta_D(t, \overline{h}) < 0$ for $t \in (t_0, \infty)$, we can apply \eqref{same-cl-in} to obtain the continuity of $(f_{\overline{h}} \circ \mu)(D_{\overline{h}, t})$. 
It therefore follows from \eqref{ODE limit} that $\delta_D(\cdot, \overline{h})$ is a solution to \eqref{dist ODE} of class $C^1((0,\infty))$. 
\end{proof}

\subsection{The representation formula}
Let us prove Theorem \ref{thm:control} in this section. Recall that the functions $U_D, U_E, u_D, u_E$ are defined by \eqref{def:value}, \eqref{def:value2}, \eqref{new upper value}, \eqref{new lower value}, respectively. 
We first remark that
\begin{equation}\label{eq suboptimal add2}
U_D(x, t, u_D(x, t))\geq 0 \quad \text{for all $(x, t)\in \R^n\times [0, \infty)$}.
\end{equation}
Indeed, thanks to the left continuity of $h\mapsto \delta_D(t, h)$ shown in Lemma \ref{lem:ode}, we have 
\[
\lim_{\vep\to 0+} U_D(x, t, h-\vep)=U_D(x, t, h)
\] 
for any $(x, t)\in \R^n\times [0, \infty)$ and $h\in \R$. Then \eqref{eq suboptimal add2} follows from the definition of $u_D$ and the monotonicity of $h\mapsto U_D(x, t, h)$ stated in Proposition \ref{prop:mono}. We can similarly show that 
$U_E(x, t, u_E(x, t))\leq 0$ for all $(x, t)\in \R^n\times [0, \infty)$.

\begin{prop}[Dynamic Programming Principle (DPP)]\label{prop dpp}
Assume that {\rm(A1)--(A4)} hold and there exists $\ol{h}\in \R$ satisfying \eqref{h-threshold}. 
Let $\delta_W$ with $W=D, E$ denote the maximal and minimal solutions as in \eqref{def-sol}.  Then, 
\begin{align}
&U_D(x, t, h)=
\left\{
\begin{array}{ll}
{\min}\{U_D(y, s, h): |y-x|\leq \delta_D(t, h)-\delta_D(s, h)\}  & \text{if} \; \; h > \overline{h}\\
\max\{U_D(y, s, h): |y-x| \le \delta_D(s,h)-\delta_D(t,h)\}  &\text{if} \; \; h \le \overline{h},  
\end{array}
\right. 
\label{dpp1}\\
&
U_E(x, t, h)=
\left\{
\begin{array}{ll}
\min\{U_E(y, s, h): |y-x|\leq \delta_E(t, h)-\delta_E(s, h)\} & \text{if} \; \; h \ge \overline{h}\\
{\max}\{U_E(y, s, h): |y-x|\leq \delta_E(s, h)-\delta_E(t, h)\} &\text{if} \; \; h < \overline{h}
\end{array}
\right. \label{dpp2}
\end{align}
for all $x\in \R^n$, $t\geq s\geq 0$. 
\end{prop}

\begin{proof}
We only prove \eqref{dpp1} with $h > \overline{h}$ as we can prove the others similarly. 
Fix $x\in\R^n$, $t\geq s\geq 0$ and $h>\overline{h}$. Take any $y\in \R^n$ satisfying
\begin{equation}\label{eq dpp1}
|y-x|\leq \delta_D(t, h)-\delta_D(s, h).
\end{equation}
By \eqref{def:value}, there exists $z\in \R^n$ such that 
\begin{equation}\label{eq dpp2}
|z-y|\leq \delta_D(s, h), 
\quad\text{and} \quad U_D(y, s, h)=U_0(z, h).
\end{equation}
We combine \eqref{eq dpp1} and \eqref{eq dpp2} to deduce that
\begin{equation}\label{eq dpp3}
\delta_D(t, h)\ge \delta_D(s, h)+|y-x|\ge |z-y|+|y-x|\ge |z-x|, 
\end{equation}
which, by \eqref{def:value} again, yields
$U_D(x, t, h)\leq U_0(z, h)= U_D(y, s, h)$.
Letting $\vep\to 0$, we have
\[
U_D(x, t, h)\leq \inf\{U_D(y, s, h): |y-x|\leq \delta_D(t, h)-\delta_D(s, h)\}.
\]

Let us next prove the opposite inequality. 
Take $z\in \R^n$ satisfying \eqref{eq dpp3} 
and $U_D(x, t, h)= U_0(z, h)$. 
For any $0\leq s\leq t$, take $y\in \R^n$ such that 
\[
|z-y|+|y-x|=|z-x|, \quad |z-y|=\min\{\delta_D(s, h), |z-x|\}, 
\]
which implies \eqref{eq dpp1}. By \eqref{def:value}, we obtain $U_D(y, s, h)\leq U_0(z, h)= U_D(x, t, h)$. 
It follows that 
\[
U_D(x, t, h)\geq \inf\{U_D(y, s, h): |y-x|\leq \delta_D(t, h)-\delta_D(s, h)\}, 
\]
which concludes the proof. 
\end{proof}

\begin{prop}\label{prop:mono}
Assume that {\rm(A1)--(A4)} hold and there exists $\ol{h}\in \R$ satisfying \eqref{h-threshold}. 
Then,  for each $(x, t)\in \R^n\times [0, \infty)$, $h\mapsto U_D(x, t, h)$ and $h\mapsto U_E(x, t, h)$ 
are strictly decreasing in $\R$. 
\end{prop}
\begin{proof}
For any $\overline{h} < h_1<h_2$, $(x, t)\in \R^n\times [0, \infty)$, by Lemma \ref{lem:ode}(a), we use the definition of $U_D$ in \eqref{def:value} to get
\[
\begin{aligned}
U_D(x, t, h_1)&=\inf\{U_0(y, h_1): |y-x|\leq \delta_D(t, h_1)\}\\
&> \inf\{U_0(y, h_2)+h_2-h_1: |y-x|\leq \delta_D(t, h_2)\}\\
&> \inf\{U_0(y, h_2): |y-x|\leq \delta_D(t, h_1)\} =U_D(x, t, h_2).
\end{aligned}
\]
One can similarly show the same property for $h_1<h_2\leq \ol{h}$. For $h_1\leq \ol{h}<h_2$, we have $\delta_D(t, h_1)\leq 0\leq \delta_D(t, h_2)$. By \eqref{def:value} again, we obtain 
\[
U_D(x, t, h_1)\geq U_0(x, h_1)>U_0(x, h_2)\geq U_D(x, t, h_2).
\]
This completes the proof of strict monotonicity of $h\mapsto U_D(x, t, h)$. The proof for the monotonicity of $h\mapsto U_E(x, t, h)$ is analogous and thus omitted here. 
\end{proof} 


\begin{prop}\label{prop direct compare} 
Assume that {\rm(A1)--(A4)} hold and there exists $\ol{h}\in \R$ satisfying \eqref{h-threshold}. 
Then \eqref{bounds} holds. 
\end{prop}
\begin{proof}
We first show that $u_E\leq u_D$ in $\R^n\times [0, \infty)$. 
In view of (A2) and Lemma \ref{lem:ode}(d), by \eqref{def:value} and \eqref{def:value2} we easily see that 
\begin{equation}\label{height comparison}
U_D(x, t, h)\geq U_E(x, t, h)
\end{equation}
for all $(x, t)\in \R^n\times [0, \infty)$ and $h\in \R$. In view of the definition of $u_D$ in \eqref{new upper value}, for any $\vep>0$, 
we have $U_D(x, t, u_D(x, t)+\vep)< 0$, which by \eqref{height comparison} implies that $U_E(x, t, u_D(x, t)+\vep)< 0$. Applying the definition of $u_E$ in \eqref{new lower value}, we obtain $u_E\leq u_D+\vep$ in $\R^n\times [0, \infty)$, which implies the desired inequality due to the arbitrariness of $\vep>0$. 

We next show that 
\begin{equation}\label{bounds-d}
\min\{u_0(x),\ \ol{h}\} \leq u_D(x, t)  \leq \max\{u_0(x),\ \ol{h}\}; 
\end{equation}
The proof for $u_E$ is similar. If $u_0(x)\geq \ol{h}$, then by \eqref{def:value} we get 
\[
U_D(x, t, \ol{h})\geq U_0(x, \ol{h})\geq 0,\quad
U_D(x, t, u_0(x)+\vep)\leq U_0(x, u_0(x)+\vep)<0
\] 
for any $\vep>0$, which by \eqref{new upper value} implies $\ol{h}\leq u_D(x, t)\leq u_0(x)+\vep$. Sending $\vep\to 0$, we get \eqref{bounds-d} in this case. When $u_0(x)<\ol{h}$, we have for any $\vep>0$
\[
U_D(x, t, u_0(x))\geq U_0(x, u_0(x))=0, \quad U_D(x, t, \ol{h}+\vep)\leq U_0(x, \ol{h}+\vep)<0,
\]
which by \eqref{def:value} yields $u_0(x)\leq u_D(x, t)\leq \ol{h}+\vep$. We obtain \eqref{bounds-d} again by letting $\vep\to 0$. 
\end{proof}

\begin{prop}[Consistency of level and parallel sets] \label{prop:parallel-surfaces} 
Assume that {\rm(A1)--(A4)} hold and there exists $\ol{h}\in \R$ satisfying \eqref{h-threshold}. For any $h\in \R$ and $t\geq 0$, 
\begin{equation}\label{coincide1}
\{u_D(\cdot, t)<h\}=D_h(\delta_D(t, h)), 
\quad 
\{u_E(\cdot, t)\leq h\}=E_h(\delta_E(t, h)). 
\end{equation}
\end{prop}
\begin{proof}
We only consider the case $h > \overline{h}$, 
as we can analogously prove in the case $h\leq \overline{h}$.  
Fix $y\in \{u_D(y, t)<h\}$ arbitrarily. By definition 
 \eqref{new upper value} of $u_D$, we have 
$U_D(y, t, h)<0$. It follows from the definition of $U_D$ that 
\[
\min\{U_0(z, h): |z-y|\leq \delta_D(t, h)\}<0.
\]
Thus, there exists $z\in \R^n$ such that $|z-y|\leq \delta_D(t, h)$ and  $u_0(z)<h$, 
which implies that $y\in D_h(\delta_D(t, h))$. We have shown that $ \{u_D(y, t)<h\}\subset D_h(\delta_D(t, h))$.
Since all of the implications above are reversible, we thus obtain
$D_h(\delta_D(t, h))= \{u_D(\cdot, t)<h\}$. 
Similarly, we can prove $\{u_E(\cdot, t)\leq h\}=E_h(\delta_E(t, h))$. 
\end{proof}

We next prove that $u_D$ and $u_E$ are, respectively, a subsolution and a supersolution of \eqref{E1}. 
Let us first prove a sub-optimality and super-optimality principle for $u_D$ and $u_E$, respectively. 

\begin{lem}\label{lem:optimality} 
Assume that {\rm(A1)--(A4)} hold and there exists $\ol{h}\in \R$ satisfying \eqref{h-threshold}. Then the following inequalities hold for all $x\in \R^n$ and $t\geq s\geq 0$.
\begin{equation}\label{new suboptimality}
\begin{aligned}
&u_D(x, t)\leq \begin{cases}
\inf\left\{u_D(y, s): |y-x|\leq \delta_D(t, u_D(x, t))-\delta_D(s, u_D(x, t)) \right\} & \text{if} \; \; u_D(x, t) > \overline{h}, \\
\sup\left\{u_D(y, s): |y-x|\leq \delta_D(s, u_D(x, t))-\delta_D(t, u_D(x, t)) \right\} & \text{if} \; \; u_D(x, t) \leq \overline{h}.
\end{cases}
\end{aligned}
\end{equation}
\begin{equation}\label{new superoptimality}
\begin{aligned}
&u_E(x, t)\geq \begin{cases}
\inf\left\{u_E(y, s): |y-x|\leq \delta_E(t, u_E(x, t))-\delta_E(s, u_E(x, t))\right\} & \text{if} \; \; u_E(x, t) \geq \overline{h}, \\
\sup\left\{u_E(y, s): |y-x|\leq \delta_E(t, u_E(x, t))-\delta_E(s, u_E(x, t)) \right\} & \text{if} \; \; u_E(x, t) < \overline{h} 
\end{cases}
\end{aligned}
\end{equation}
\end{lem}

\begin{proof}

 We only prove \eqref{new suboptimality} as we can prove \eqref{new superoptimality} in a symmetric manner. Fix $t\geq s\geq 0$ and $x\in \R^n$. Suppose that $u_D(x, t) > \overline{h}$. By Proposition \ref{prop dpp} with $h=u_D(x, t)$, we have 
\[
U_D(x, t, u_D(x, t))\leq U_D(y, s, u_D(x, t)).
\]
In view of \eqref{eq suboptimal add2}, we obtain $U_D(y, s, u_D(x, t))\geq 0$ for all $y\in \R^n$ satisfying 
\[
|y-x|\leq \delta_D(t, u_D(x, t))-\delta_D(s, u_D(x, t)),
\]
which immediately implies \eqref{new suboptimality} in the case that $u_D(x, t) > \overline{h}$. 
We may apply the same argument when $u_D(x, t)\leq \overline{h}$. In this case, we use \eqref{eq suboptimal add2} to get $U_D(y, s, u_D(x, t))\geq 0$ for some $y\in \R^n$ satisfying $|y-x|\leq \delta_D(s, u_D(x, t))-\delta_D(t, u_D(x, t))$, which again yields \eqref{new suboptimality}.
\end{proof}

\begin{prop}[Subsolution and supersolution properties]\label{prop main} 
Assume that {\rm(A1)--(A4)} hold and there exists $\ol{h}\in \R$ satisfying \eqref{h-threshold}. Then, $u_D$ and $u_E$ are, respectively, a subsolution and a supersolution of \eqref{E1} satisfying $u_D(\cdot, 0)= u_E(\cdot, 0)=u_0$ on $\R^n$.
\end{prop}
\begin{proof}
 By the definitions of $u_D$, $u_E$, we can easily verify the initial condition $u_D(\cdot, 0)= u_E(\cdot, 0)=u_0$ on $\R^n$. Let us prove that $u_D$ is a subsolution of \eqref{E1}. The proof for the part with $u_E$ is similar. 
 
We first show that $u_D\in \USC(\R^n\times[0,\infty))$.  
Fix $(x_0, t_0)\in \R^n\times [0, \infty)$ arbitrarily. 
Assume that $(x_j, t_j)\in \R^n\times [0, \infty)$ is a sequence satisfying $(x_j, t_j)\to (x_0, t_0)$ as $j\to \infty$. 
For any $\varepsilon>0$, 
by definition, there exists $j\geq 1$ large such that 
\begin{equation}\label{eq usc2}
u_D(x_j, t_j)\geq \limsup_{j\to \infty} u_D(x_j, t_j)-\vep.
\end{equation}
By \eqref{eq suboptimal add2}, we have
$U_D(x_j, t_j, u_D(x_j, t_j))\geq 0$. 
Since $h\mapsto U_D(x, t, h)$ is decreasing, it then follows from \eqref{eq usc2} that
\[
U_D\left(x_0, t_0, \limsup_{j\to \infty} u_D(x_j, t_j)-\vep\right)\geq U_D\left(x_j, t_j, 
u_D(x_j, t_j)\right)\geq 0.
\]
Using the definition of $u_D$ in \eqref{new upper value}, we obtain
$u_D(x_0, t_0)\geq \limsup_{j\to \infty} u_D(x_j, t_j)-\vep$.
Letting $\vep\to 0$, we end up with
$\limsup_{j\to \infty} u_D(x_j, t_j)\leq u_D(x_0, t_0)$, 
which completes the proof of upper semicontinuity of $u_D$ due to the arbitrariness of $(x_0, t_0)$. 

We next verify the subsolution property of $u_D$. 
Take any $(x_0, t_0)\in \R^n\times (0, \infty)$ and $\varphi\in C^1(\R^n\times (0, \infty))$ such that $u_D-\varphi$ attains a maximum at $(x_0, t_0)$ 
with $(u_D-\varphi)(x_0,t_0)=0$. 
We first consider the case $u_D(x_0, t_0) > \overline{h}$. 
Note that by \eqref{new suboptimality} in Lemma \ref{lem:optimality} we have
\[
\varphi(x_0, t_0)=u_D(x_0, t_0)
\leq \inf\{u_D(y, s): |y-x|\leq \delta_D(t_0, u_D(x_0, t_0))-\delta_D(s, u_D(x_0, t_0))\} 
\]
for all $0\leq s\leq t_0$. 
For any $(y,s)\in\R^n\times(0,t_0]$ satisfing 
\[
|y-x_0|\leq \delta_D(t_0, u_D(x_0, t_0))-\delta_D(s, u_D(x_0, t_0)), 
\]
we have 
$\varphi(x_0, t_0)=u_D(x_0, t_0)\le u_D(y,s)\le \varphi(y,s)$, which implies 
\[
0\le \varphi(y,s)-\varphi(x_0,t_0)=D\varphi(x_0,t_0)\cdot(y-x_0)+\varphi_t(x_0,t_0)(s-t_0)+o(t_0-s). 
\]
We can take such $y\in\R^n$ additionally satisfying
$D\varphi(x_0,t_0)\cdot(y-x_0)=-|D\varphi(x_0,t_0)||y-x_0|$. 
Then, by Taylor expansion we obtain 
\[
0\leq-|\nabla \varphi(x_0, t_0)|(\delta_D(t_0, u_D(x_0, t_0))-\delta_D(s, u_D(x_0, t_0)))+\varphi_t(x_0, t_0) (s-t_0)+o(t_0-s),
\]
which implies
\[
\varphi_t(x_0, t_0) +|\nabla \varphi(x_0, t_0)|\frac{\delta_D(t_0, u_D(x_0, t_0))-\delta_D(s, u_D(x_0, t_0))}{t_0-s}\leq o(1)
\]
for all $s< t_0$ close to $t_0$. 
Taking $h_0=u_D(x_0, t_0)$ and letting $s\to t_0-$, by \eqref{dist ODE} we get
\[
\varphi_t(x_0, t_0) +|\nabla \varphi(x_0, t_0)|f(u_D(x_0, t_0), \mu(D_{h_0}(\delta_D(t_0, h_0))))\leq 0.
\]
Noticing that, by \eqref{coincide1}, 
$D_{h_0}(\delta_D(t_0, {h_0}))=\{u(\cdot, t_0)<u(x_0, t_0)\}$
holds, we thus obtain \eqref{sub eq} in the case that $u_D(x_0, t_0) >\overline{h}$.  
If $u_D(x_0, t_0) \leq \overline{h}$, we can prove \eqref{sub eq} similarly. In this case we adopt \eqref{new suboptimality} to deduce that, for $0\leq s<t_0$,
\[
\varphi(x_0, t_0)=u_D(x_0, t_0)
\leq \sup\{u_D(y, s): |y-x|\leq \delta_D(s, u_D(x_0, t_0))-\delta_D(t_0, u_D(x_0, t_0))\} 
\]
Then there exists $y\in \R^n$ such that 
\[
 |y-x|\leq \delta_D(s, u_D(x_0, t_0))-\delta_D(t_0, u_D(x_0, t_0))
\]
and, by Taylor expansion, 
\[
0\leq |\nabla \varphi(x_0, t_0)|(\delta_D(s, u_D(x_0, t_0))-\delta_D(t_0, u_D(x_0, t_0)))+\varphi_t(x_0, t_0) (s-t_0)+o(t_0-s),
\]
Dividing the inequality above by $t_0-s$ and sending $s\to t_0-$, we obtain \eqref{sub eq} again. 
\end{proof} 

In view of \eqref{bounds} and the uniform continuity of $u_0$ in (A4), we easily obtain the growth condition \eqref{eq:growth} for any $T>0$ with both $u=u_D$ and $u=u_E$. Therefore, adopting Proposition \ref{prop main} with the comparison principle, Theorem \ref{thm:comparison1}, we deduce that $u_D\leq u_E$ in $\R^n\times [0, \infty)$. Since by Proposition \ref{prop direct compare}, we have $u_E\leq u_D$ in $\R^n\times [0, \infty)$. Then Theorem \ref{thm:control} holds as an immediate consequence.

\section{Fattening and non-fattening results}\label{sec:fat}
In this section we study the behavior of level sets of the solution to \eqref{E1}, \eqref{initial}. We are interested in the problem whether the level set of the solution develops interior during the evolution. Let $h\in\R$. We recall that $\delta_D$ and $\delta_E$ are the functions given by \eqref{def-sol}, and $E_h(s), D_h(s)$ are the parallel sets given by \eqref{parallel set} for $s\in\R$. 
Thanks to Proposition \ref{prop:parallel-surfaces}, the $h$-level set of $u$ for $h\in \R$ can be represented by 
\begin{equation}\label{eq level-set}
\{x \in \mathbb{R}^n: u(x,t) = h\} = E_h(\delta_E(t, h)) \setminus D_h(\delta_D(t,h)). 
\end{equation}
The fattening and non-fattening phenomena can thus be analysed through the behavior of $\delta_E$ and $\delta_D$. 
%
%
%
Let us first give a proof of Theorem \ref{thm:non-th-fatten}. 
\begin{proof}[Proof of Theorem \ref{thm:non-th-fatten}]
(a) We only prove in the case $h > \overline{h}$, since a similar argument can be applied to the case $h<\ol{h}$. 
Let us begin with an elementary topological fact: 
\begin{equation}\label{close-D}
\cl{D_h(s)} = \{x \in \mathbb{R}^n : |y - x| \le s \; \text{for some $y \in \cl{D_h(0)}$} \} \quad \text{if} \; \; s \ge 0.
\end{equation}
Indeed, if there exists $x_j\in D_h(s)$ such that $x_j \to x_0$ as $j\to \infty$, since we can find $y_j\in D_h(0)$ such that $|x_j-y_j|\leq s$,  by taking a subsequence we have $y_j\to y_0$ for some $y_0\in \cl D_h(0)$ satisfying $|y_0-x_0|\leq s$. It follows that $x_0\in \cl D_h(s)$. 


 Let $\delta_D(\cdot, h)$ and $\delta_E(\cdot, h)$ be respectively the minimal and maximal solutions obtained in Proposition \ref{prop:sol-threshold} with $W=D$ and $W=E$, as redefined in \eqref{def-sol}. Under the condition \eqref{cond:nonfat}, we prove that $\delta_D(\cdot, h)=\delta_E(\cdot, h)$ in $[0, \infty)$ and they are respectively the only solutions to \eqref{dist ODE} for $W=D$ and $W=E$.
 Since $\delta_D(t, h)$ is positive for $t >0$, $D_{h, t}$ has finite perimeter for $t>0$ due to \eqref{same-cl-in}. Here we recall the notation $W_{h, t}=W_h(\delta_W(t,h))$ for $W=D, E$ as in \eqref{abbrev}. It follows from \eqref{cond:nonfat}  and \eqref{close-D}  that $\mu(\cl D_{h, t})=\mu(E_h(\delta_D(t,h)))$ for all $t>0$.
Therefore, by Lemma \ref{lem:reform-ODE} we get
\[ (\delta_D)_t(t,h) = (f_h \circ \mu) (\cl{D_{h,t}}) = (f_h \circ \mu)(E_h(\delta_D(t,h))) \quad \text{for all }  t > 0. \]
These are combined to show that $\delta_D(\cdot, h)$ is actually a solution to \eqref{dist ODE} with $W=E$ as well. 

Let $\tilde{\delta}_D(\cdot, h)$ be a solution to \eqref{dist ODE} with $W=D$ possibly different from $\delta_D(\cdot, h)$. Since $\delta_D$ is a solution to \eqref{dist ODE} with $W=E$, by a similar argument for Lemma~\ref{lem:ode}~(d), we have 
$\tilde{\delta}_D(t, h) \le \delta_D(t, h)$ for $t \ge 0$. 
Since $\delta_D$ is the minimal solution to \eqref{dist ODE}, we get $\delta_D(t , h)= \tilde{\delta}_D(t, h)$ for all $t\geq 0$ and thus obtain the uniqueness.

It can be proved similarly that $\delta_E(\cdot, h)$ is also a solution to \eqref{dist ODE} with $W=D$ and therefore the uniqueness implies $\delta_D(\cdot,h) = \delta_E(\cdot,h)$ in $[0, \infty)$. By using \eqref{close-D} again, we have 
$\cl{D_h(\delta_D(t, h))} = E_h(\delta_E(t, h))$, 
which yields the conclusion by \eqref{eq level-set}. 

We omit the proof in the case that $h<\ol{h}$, as it is similar to our proof above. Instead of \eqref{close-D}, we use the following relation for this case:
\[
\inter E_h(s) = \{x \in \mathbb{R}^n: |y-x| \le -s\, \; \text{for any $y \in \inter E_h(0)$}\} \quad \text{if} \; \; s \le 0.
\]
(b) By assumption \eqref{as:fattening}, there exists $x \in E_h(0)\setminus \cl{D}_h(0)$, which yields, for a small $r>0$, 
$B_{2r}(x) \cap D_h(0) = \emptyset$.
Here $B_{2r}(x)$ denotes the open ball with center $x$ and radius $2r$. It follows from (A3) that $\mu(D_h(r))<\mu(E_h(r))$.
Since $\delta_D(t,h) \le \delta_E(t,h)$ holds for all $t \ge 0$, taking $T_r>0$ that satisfies $\delta_E(t,h) < r$ for all $0 \le t < T_r$, we have 
\begin{equation}\label{fattening-dife-t} 
(\delta_E)_t(t,h) = (f_h \circ \mu)(E_{h,t}) > (f_h \circ \mu)(D_{h, t}) = (\delta_D)_t(t,h) \quad \text{for} \; \; 0 < t < T_r, 
\end{equation}
which implies $\delta_E(t,h) > \delta_D(t,h)$ for all $0 < t < T_r$. 
Let us next show that $\delta_E(t,h) > \delta_D(t,h)$ holds for all $t>0$. Assume by contradiction that 
\[ T' := \sup\{T > 0: \delta_E(t,h) > \delta_D(t,h) \; \text{holds for $0<t<T$}\}<\infty. \] 
We then have $\delta_E(T',h) = \delta_D(T',h)$ and $\delta_E(t, h) > \delta_D(t,h)$ for $0<t<T'$, 
which yields the inequality \eqref{fattening-dife-t} but with $T_r$ replaced by $T'$. 
Integrating it over the interval $[0,T']$, we have $\delta_E(T',h) > \delta_D(T',h)$, which contradicts the fact that $\delta_E(T',h) = \delta_D(T',h)$. 
\end{proof}

We present a concrete example in one space dimension for the statement (b) in Theorem \ref{thm:non-th-fatten}. 
\begin{ex}\label{ex partial fattening}
Suppose that $\mu$ is a measure in $\R$ satisfying (A3) such that $\mu(A)=m(A)$ for any measurable set $A\subset [-10, 10]$. Let $F: \R\to \R$ be a bounded, continuous, strictly increasing function such that $F(q)=q-1$ for all $0\leq q\leq 10$. Consider the nonlocal equation 
\[
u_t+|\nabla u| F(\mu(\{u(\cdot, t)<u(x, t)\}))=0\quad \text{in $\R\times (0, \infty)$}
\]
with the initial value
\[
u(x, 0)=u_0(x)=\min\{|x+1|+1, |x-2|\},\quad x\in \R. 
\]
It is clear that $\ol{h}=1/2$, since $m(\{u_0<1/2\})=m(\{u_0\leq 1/2\})=1$ and $F(1)=0$. We focus on the evolution of the sublevel set at level $h=1$. 

Let us adopt our control-theoretic interpretation to determine the value of the solution at $(x, t)\in \cN_\vep \times [0, \sigma)$, where we denote by $\cN_\vep$ the $\vep$-neighborhood of $\{-1\}\cup [1, 3]$ and take $\vep, \sigma>0$ small. Note that in this case, we have 
\[
D_h(0)=\begin{cases}
(-1-h, -1+h)\cup (1-h, 3+h) & \text{if $h>1$,}\\
(1-h, 3+h) & \text{if $h\leq 1$, }
\end{cases}
\]
\[
\quad E_h(0)=\begin{cases}
[-1-h, -1+h]\cup [1-h, 3+h] & \text{if $h\geq 1$,}\\
[1-h, 3+h] & \text{if $h< 1$, }
\end{cases}
\]
and therefore 
\[
(\delta_D)_t(t, h)= \begin{cases}
4\delta_D(t, h)+4h-3 & \text{if $h>1$,}\\
2\delta_D(t, h)+2h-1 & \text{if $h\leq 1$,}
\end{cases}
\quad 
(\delta_E)_t(t, h)= \begin{cases}
4\delta_E(t, h)+4h-3 & \text{if $h\geq1$,}\\
2\delta_E(t, h)+2h-1 & \text{if $h< 1$,}
\end{cases}
\]
when $|h-1|$ and $t\geq 0$ are both sufficiently small. For such $h$ and $t$, a direct computation with the initial data $\delta_D(0, h)=\delta_E(0, h)=0$ yields that 
\[
\delta_D(t, h)=\begin{cases}
(e^{4t}-1)\left(h-{3\over 4}\right) & \text{if $h>1$,}\\
(e^{2t}-1)\left(h-{1\over 2}\right) & \text{if $h\leq 1$,}
\end{cases}
 \quad 
 \delta_E(t, h)=\begin{cases}
(e^{4t}-1)\left(h-{3\over 4}\right) & \text{if $h\geq1$,}\\
(e^{2t}-1)\left(h-{1\over 2}\right) & \text{if $h< 1$,}
\end{cases}
\]
which indicates the discrepancy between $\delta_D(t, 1)$ and $\delta_E(t, 1)$, that is, for $t>0$ small, 
\[
\delta_D(t, 1)={1\over 2} e^{2t}-{1\over 2}< {1\over 4}e^{4t}-{1\over 4}=\delta_E(t, 1).
\]
This is consistent with our result in Theorem \ref{thm:non-th-fatten}(b). We can further obtain that
\[
U_D(x, t, h)=\begin{cases}
u_0(x)-e^{4t}\left(h-{3\over 4}\right)-{3\over 4} & \text{if $h>1$,}\\
u_0(x)-e^{2t}\left(h-{1\over 2}\right)-{1\over 2} & \text{if $h\leq 1$,}
\end{cases}
\]
\[
U_E(x, t, h)=\begin{cases}
u_0(x)-e^{4t}\left(h-{3\over 4}\right)-{3\over 4} & \text{if $h\geq 1$,}\\
u_0(x)-e^{2t}\left(h-{1\over 2}\right)-{1\over 2} & \text{if $h< 1$.}
\end{cases}
\]
Then, using our definitions of $u_D$ and $u_E$ in  \eqref{new upper value}\eqref{new lower value}, for $t\in [0, \sigma)$
 with $\sigma>0$ small we get
 $u_D(x, t)\geq 1$ for all $x\in \R\setminus \cN_{\delta_D(t, 1)}$ and $u_E(x, t)\leq 1$ for all $x\in \cN_{\delta_E(t, 1)}$. It follows that $u(x, t)=u_D(x, t)=u_E(x, t)=1$ for all $x\in \cN_{\delta_E(t, 1)}\setminus \cN_{\delta_D(t, 1)}$ with $t>0$ small. 
 
 Strictly speaking, this example cannot be classified as the fattening phenomenon, since $\cl D_h(0)\neq E_h(0)$ for $h=1$. However, if we narrow our focus to the evolution of $\{u(\cdot, t)=1\}$ near the boundary points $x=1$ and $x=3$ of $[1, 3]$, we do observe the instantaneous development of interior. Such behavior is due to the measure-dependent normal velocity as well as the discontinuity of the increment rate of $h\mapsto \mu(\{u_0<h\})$ at $h=1$. 
\end{ex}

Let us now turn to the evolution of the $\overline{h}$-level set.  We begin with a non-fattening result.

\begin{thm}\label{prop:fattening-o-h-1}
Assume that (A1)--(A4) hold and the function $\Theta$ in (A3) is bounded in $\R^n$. Assume that there exists $\ol{h}\in \R$ satisfying \eqref{h-threshold}. Assume that
$f(\overline{h}, \cdot)$ is of class ${\rm Lip}_{\rm loc}([0, \infty))$. Assume also that 
\begin{equation}\label{as-same-measure} 
\mu(\{x \in \mathbb{R}^n: u_0(x) < \overline{h}\}) = \mu(\{x \in \mathbb{R}^n: u_0(x) \le \overline{h}\}) 
\end{equation}
and there exist constants $C, s_0> 0$ such that 
\begin{equation}\label{as-bdd-peri} 
{\rm Per}(D_{\overline{h}}(-s))+{\rm Per}(E_{\overline{h}}(s)) \le C \quad \text{for} \; \; 0 < s \le s_0. 
\end{equation}
Then, $\delta_W(t, \overline{h}) \equiv 0$ for $t \ge 0$, in other words, for any $t \ge 0$,  
\[
\{x \in \mathbb{R}^n: u(x,t) < \overline{h}\} = \{x \in \mathbb{R}^n: u_0 < \overline{h}\}, \quad 
\{x \in \mathbb{R}^n: u(x,t) \le \overline{h}\} = \{x \in \mathbb{R}^n: u_0 \le \overline{h}\}. 
\]
\end{thm}
\begin{proof}
Noting that \eqref{as-same-measure} implies
$(f_{\overline{h}} \circ \mu)(D_{\overline{h}}(0)) = (f_{\overline{h}} \circ \mu)(E_{\overline{h}}(0)) = 0$, 
we observe that $\delta(\cdot, \overline{h}) \equiv 0$ is a solution to \eqref{dist ODE}, respectively, for $W=D$ and $W=E$. 
On the other hand, nonpositive solutions to \eqref{dist ODE} with $W=D$ and nonnegative solutions to \eqref{dist ODE} with $W=E$ are unique. 
Indeed, due to the boundedness of $\Theta$ in (A3),  the co-area formula (see \cite[Theorem 3.10]{EG} for instance), we have 
\begin{align*} 
| \mu (D_{\overline{h}}(-s_1)) - \mu(D_{\overline{h}}(-s_2))| \le \Theta_{\max} \int_{s_1}^{s_2} {\rm Per}(D_{\overline{h}}(-s)) \; ds \quad \text{for} \; \; 0 \le s_1 < s_2,  
\end{align*}
where $\Theta_{\max} := \sup_{x \in \R^n} \Theta(x)$.
Therefore, due to assumption \eqref{as-bdd-peri} and the Lipschitz continuity of $f(\overline{h}, \cdot)$, we see that $f_{\overline{h}} \circ \mu (D_{\overline{h}}(-\cdot))$ is locally Lipschitz continuous on $[0, \infty)$. 
In light of the unique existence theory for ordinary differential equation, the uniqueness of non-positive solutions to \eqref{dist ODE} with $W=D$ can be proved. 
We can similarly prove the uniqueness of non-negative solution to \eqref{dist ODE} with $W=E$. 
Since $\delta(\cdot, \overline{h}) \equiv 0$ is a solution to \eqref{dist ODE}, we thus obtain $\delta_W(\cdot, \overline{h}) \equiv 0$ in $[0, \infty)$ respectively for $W=D$ and $W=E$. 
\end{proof}

The next proposition provides a sufficient condition for the occurrence of the fattening phenomenon at the $\ol{h}$-level.

\begin{thm}
Assume that {\rm(A1)--(A4)} hold and there exists $\ol{h}\in \R$ satisfying \eqref{h-threshold}. If 
\begin{equation}\label{fattening suff1}
 \liminf_{s\to 0+} \frac{f(\ol{h}, \mu(D_{\ol{h}}(-s)))}{s^\alpha}\leq -C\quad \text{for some $0<\alpha<1$ and $C>0$},
\end{equation}
then $\delta_D(t, \overline{h})<0\leq \delta_E(t, \overline{h})$ holds for all $t > 0$. If 
\begin{equation}\label{fattening suff2}
 \limsup_{s\to 0+}\frac{f(\ol{h}, \mu(E_{\ol{h}}(s)))}{s^\alpha}\geq C \quad \text{for some $0<\alpha<1$ and $C>0$},
\end{equation}
then $\delta_D(t, \overline{h})\leq 0< \delta_E(t, \overline{h})$ holds for all $t > 0$. In particular, for the unique solution $u$ of \eqref{E1} and \eqref{initial}, its $\ol{h}$-level set $E_{\ol{h}}(t)\setminus D_{\ol{h}}(t)$ contains nonempty interior for all $t>0$ if either \eqref{fattening suff1} or \eqref{fattening suff2} holds.
\end{thm}
\begin{proof}
Assume that \eqref{fattening suff1} holds, then there exist $s_0>0$ and $0<C'<C$ such that  $f(\ol{h}, \mu(D_{\ol{h}}(-s)))\leq -C's^\alpha$ holds for all $0<s<s_0$. Setting
\[ 
\delta(t) := - \left(\frac{C'(1-\alpha)}{\alpha}\right)^{1\over 1-\alpha} t^{1\over 1-\alpha} \quad \text{for} \; \; t \geq 0,
\]
we have $ \delta_t(t) = -C'(-\delta(t))^{\alpha}/\alpha$ for all $t>0$. 
Then by the choice of $\delta_D$ and $\delta_E$, there exists $T>0$ small such that $\delta_D(t, \overline{h}) \le \delta(t) < 0 \le \delta_E(t, \overline{h})$ for all $0 < t < T$. 
We then see that $t \mapsto \delta_D(t, \overline{h})$ is strictly decreasing for $t \ge T$, which yields $\delta_D(t, \overline{h}) < 0 \le \delta_E(t, \overline{h})$ for any $t>0$. If, on the other hand, \eqref{fattening suff2} holds, we can show that $\delta_D(t, \overline{h}) \le0<  \delta_E(t, \overline{h})$ holds for all $0< t < T$ 
in a similar way.
\end{proof}

It is clear that either of \eqref{fattening suff1} and \eqref{fattening suff2} holds when 
\[
\mu(\{x \in \mathbb{R}^n: u_0(x) < \overline{h}\}) < \mu(\{x \in \mathbb{R}^n: u_0(x) \le \overline{h}\}).
\]
This case can hardly be considered as fattening, since the $\ol{h}$-level set of the initial value $u_0$ already has nonempty interior.

On the other hand, the general conditions \eqref{fattening suff1} \eqref{fattening suff2} do include the situation when the level sets of $u_0$ have empty interior so that
$\cl {D}_{\ol{h}}(0)=E_{\ol{h}}(0)$, $D_{\ol{h}}(0)=\inter E_{\ol{h}}(0)$ and
\begin{equation}\label{eq:nonfat-any}
\mu(E_h(0)\setminus D_h(0))=0 \quad\text{for all} \ h\in\R.
\end{equation} 
 The following result clarifies the conditions in terms of more precise regularity assumptions on $f$ and $\mu$ as well as certain growth rate conditions of the perimeters of $D_{\ol{h}}$ and $E_{\ol{h}}$. 

\begin{thm}\label{prop:fattening-o-h-3}
Assume that (A1)--(A4) hold and there exists $\ol{h}\in \R$ satisfying \eqref{h-threshold}. Assume also that for any bounded set $A\subset \R^n$, there exists $\Theta_{A}> 0$ such that 
\begin{equation}\label{eq a5}
\inf_{x \in A} \Theta(x) \ge \Theta_{A}.
\end{equation}
Suppose that $f(\overline{h}, \cdot) \in C^1([0, \infty))$. 
If there exist $C, s_0>0$ and $0<\sigma<1$ such that 
\begin{equation}\label{as-fattening-b} 
f_q(\overline{h}, \mu(D_{\overline{h}}(0))) > 0, 
\quad
{\rm Per}(D_{\overline{h}}(s)) \ge \frac{C}{(-s)^{\sigma}} \quad \text{for all } -s_0 \le s <0,  
\end{equation}
then \eqref{fattening suff1} holds. If there exist  $C, s_0>0$ and $0<\sigma<1$ such that 
\begin{equation}\label{as-fattening-c} 
f_q(\overline{h}, \mu(E_{\overline{h}}(0))) > 0, 
\quad
{\rm Per}(E_{\overline{h}}(s)) \ge \frac{C}{s^{\sigma}} \quad \text{for all } 0 < s \le s_0, 
\end{equation}
then  \eqref{fattening suff2} holds. 
\end{thm}

\begin{proof}
Suppose that \eqref{as-fattening-b} holds. By \eqref{eq a5}, \eqref{as-fattening-b} and the co-area formula, 
\begin{equation}\label{ineq:mu}
\mu(D_{\overline{h}}(0))-\mu(D_{\overline{h}}(s))\ge 
\Theta_{D_{\overline{h}}(0)}\int_{s}^{0} {\rm Per}(D_{\overline{h}}(\tau)) \; d\tau \ge 
\Theta_{D_{\overline{h}}(0)}\int_{s}^{0} \frac{C}{(-\tau)^{\sigma}} \; d\tau
=\frac{C \Theta_{D_{\overline{h}}(0)}}{1-\sigma}(-s)^{1-\sigma}  
\end{equation}
for $-s_0\le s < 0$, where $\Theta_{D_{\overline{h}}(0)}= \inf_{x \in D_{\overline{h}}(0)} \Theta(x)>0$ due to \eqref{eq a5}.  Since $f_q(\overline{h}, \mu(D_{\overline{h}}(0))) > 0$, there exist $M, \ol{q}> 0$ such that 
\[
f(\overline{h}, \mu(D_{\overline{h}}(0)))-f(\overline{h}, q)
\ge M( \mu(D_{\overline{h}}(0))-q)
\]
for all $\mu(D_{\overline{h}}(0))-\overline{q}\le q\le\mu(D_{\overline{h}}(0))$. 
Noticing that $f(\overline{h}, \mu(D_{\overline{h}}(0)))\le0$, by \eqref{ineq:mu} we obtain 
\[ 
(f_{\overline{h}} \circ \mu)(D_{\overline{h}}(s)) 
\le 
M\big( \mu(D_{\overline{h}}(s))-\mu(D_{\overline{h}}(0))\big) 
\le 
- \frac{CM \Theta_{D_{\overline{h}}(0)}}{1-\sigma}(-s)^{1-\sigma}
\]
for all $-s_0 \le s < 0$, which yields \eqref{fattening suff1} immediately with $\alpha=1-\sigma$. One can similarly show \eqref{fattening suff2} under the condition \eqref{as-fattening-c}.
\end{proof}

\begin{ex}\label{ex critical fattening}
Based on \cite[Example 1]{Kr}, we present an example, where \eqref{eq:nonfat-any} holds while there exist $C, s_0>0$ and $\sigma\in(0,1)$ such that $\Per(E_{\overline{h}}(0))<\infty$ and ${\rm Per}(E_{\overline{h}}(s)) \ge C/s^{\sigma}$ for $0 < s \le s_0$. Set 
$D:=[0,2]\times[0,1]\subset\R^2$. 
We split $D$ into a family of vertical strips with width $2^{-k}$ as follows,  
\[ 
D = \bigcup_{k=0}^\infty \left[\sum_{i=0}^k 2^{-i}, \sum_{i=0}^{k+1} 2^{-i}\right] \times [0,1]. 
\]
Furthermore, for every $k=0, 1, \ldots$, divide each rectangle 
$\left[\sum_{i=0}^k 2^{-i}, \sum_{i=0}^{k+1} 2^{-i}\right] \times[0,1]$ into 
$N_k$ squares of size $l_k\times l_k$, where $N_k:=8^k$ and $l_k:=4^{-k}$. 
Let $r_k$ $x_k^j$ be the center of each square for $k\in\N\cup\{0\}$, and $j=1,2, \cdots, N_k$. 
We set 
\[ 
\Omega := \left(\bigcup_{k=0}^\infty \bigcup_{j=1}^{N_k} B_{r_k}(x_k^j)\right)\cup B_{\frac{1}{28}}((-1,0)). 
\]
See Figure \ref{fig:2}.

 \begin{figure}[H]
  \centering
  \includegraphics[height=3.8cm]{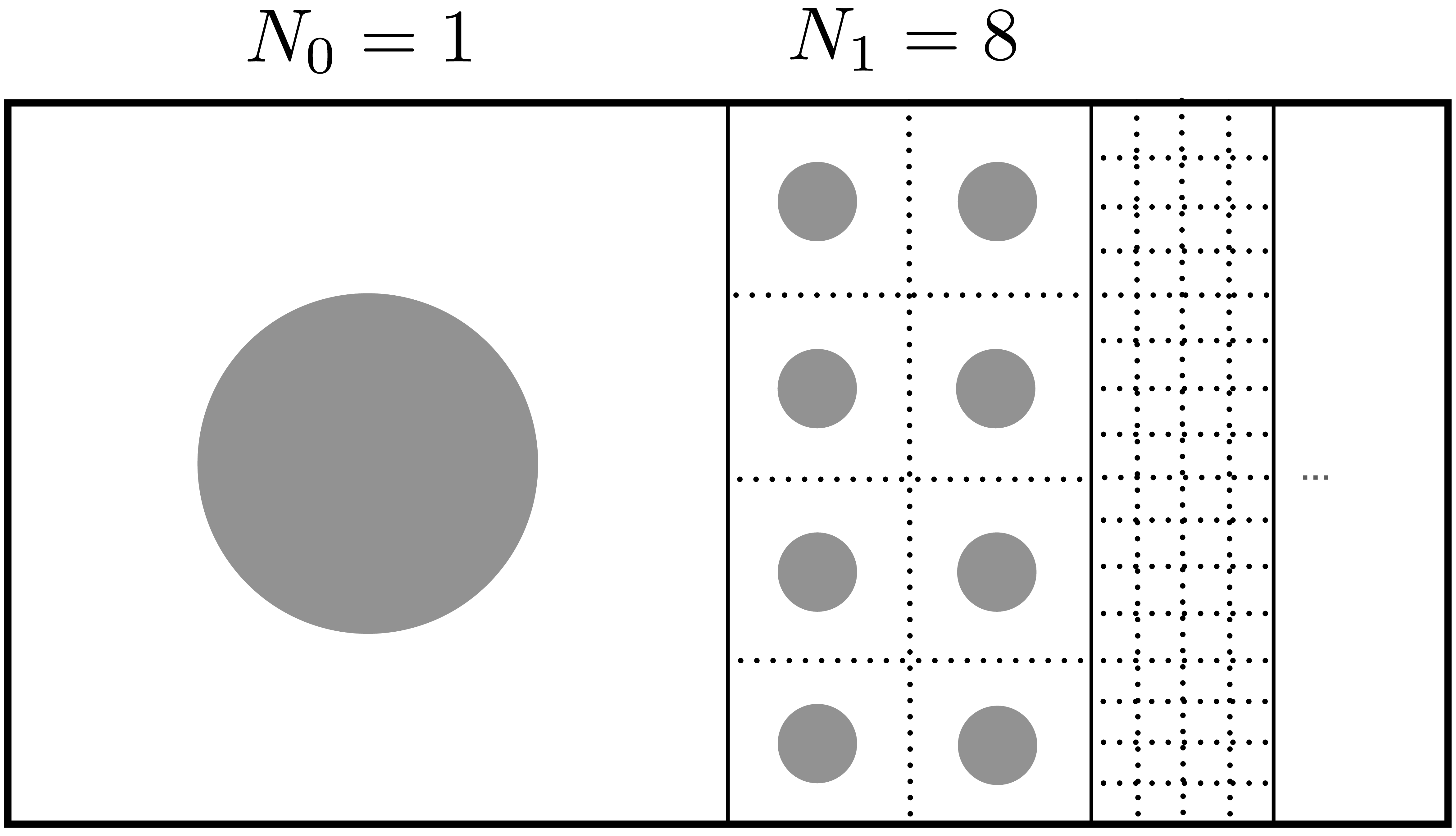}
  \caption{This figure depicts an initial level set $\Omega$ that causes fattening instantly. }\label{fig:2}
\end{figure}
Now, fix $\overline{h}\in\R$, and take a function $f:\R\times[0,\infty)\to\R$ so that 
\[
f(\overline{h}, \mu(\Omega))=0, \ 
f(h,\mu(\Omega))>0 \ \text{if} \ h>\overline{h}, \ 
\text{and} \ 
f(h,\mu(\Omega))<0 \ \text{if} \ h<\overline{h}. 
\]
Define $u_0:\R^2\to\R$ by 
\[ 
u_0(x) := 
\begin{cases}
\overline{h}+{\rm dist}(x,\Omega) & \text{if} \; \; x \in\R^2\setminus\Omega, \\
\overline{h}-{\rm dist}(x, \R^2\setminus\Omega) & \text{if} \; \; x \in\Omega. 
\end{cases} 
\]
Then, it is easily seen that $u_0$ satisfies \eqref{eq:nonfat-any}, 
$\Per(E_{\overline{h}}(0))=\frac{15\pi}{14}<\infty$, 
and 
$\overline{\Omega}=E_{\overline{h}}(0)=\{u_0\le\overline{h}\}$. 
By the argument in \cite[Example 1]{Kr}, we have, for all $s>0$ small, 
\[
\Per(E_{\overline{h}}(s))\ge 
\frac{\pi}{14}\left(\frac{1}{2\sqrt{s}}-1\right)+\frac{\pi}{14}=\frac{\pi}{28\sqrt{s}}.
\]
\end{ex}

\section{Large-time asymptotics}\label{sec:large}
In this section we investigate the large time behavior of solutions of \eqref{E1}, \eqref{initial} 
under assumptions (A1)--(A4). We also discuss the rate of convergence. 

\subsection{Convergence and finite time stabilization}

First, we prove the locally uniform convergence \eqref{eq:conv}.  
\begin{prop}\label{prop:conv}
Assume that (A1)--(A4) hold and there exists $\ol{h}\in \R$ satisfying \eqref{h-threshold}. Let $u$ be the unique solution of \eqref{E1} and \eqref{initial}. 
Then for any compact set $K\subset \R^n$, $u(\cdot, t)\to\overline{h}$ uniformly in $K$ as $t\to\infty$. 
\end{prop}
\begin{proof}
Consider first $x \in \mathbb{R}^n$ such that $u_0(x)\geq \ol{h}$. By \eqref{bounds} we have 
\begin{equation}\label{eq conv1}
u_D(x, t)\geq \ol{h} \quad \text{for all $t\geq 0$.}
\end{equation}  
Moreover,  in view of \eqref{new suboptimality}, for any $0\leq t_1\leq t_2$, we deduce that  
$u_D(x, t_1)\geq u_D(x, t_2)$. In particular, $u_D(x, t)=\ol{h}$ for all $t\geq 0$ provided that $u_0(x)=\ol{h}$.
If $u_0(x) > \overline{h}$, then since
\[ (\delta_D)_t(t, h) \ge f(h, \mu(D_h(0))) > 0 \quad \text{for} \; \; t \ge 0, \; h > \overline{h}, \]
we see that $\delta_D(t, h)\to \infty$ as $t \to \infty$ for any $h > \overline{h}$, which implies 
that there exists $t_0>0$ such that $x \in D_h(\delta_D(t,h))$ for any $t>t_0$ and 
$\limsup_{t \to \infty} u_D(x,t) \le h$ for $h > \overline{h}$. 
By \eqref{eq conv1}, we thus immediately obtain pointwise convergence of $u(\cdot, t)=u_D(\cdot, t)$ to $\ol{h}$ in $\{u_0\geq \ol{h}\}$ as $t\to \infty$. In addition, the monotonicity of $u_D$ in time enables us to use Dini's theorem to show that for every compact set $K\subset \R^n$, $u(\cdot, t)=u_D(\cdot, t)$ converges uniformly to $\ol{h}$ in $K\cap \{u_0\geq \ol{h}\}$ as $t\to \infty$.
We can apply the same argument to prove that $u(\cdot, t)=u_E(\cdot, t)$ converges uniformly to $\ol{h}$ in $K\cap \{u_0\leq \ol{h}\}$ for each compact set $K\subset \R^n$. Combining both cases, we obtain the uniform convergence of $u(\cdot, t)\to \ol{h}$ in $K$ as $t\to \infty$.
\end{proof}

We are interested in more precise behavior of the solution in long time. As mentioned in the introduction, there may exist a finite moment $T>0$ after which $u$ becomes identically equal to $\ol{h}$. It turns out to be closely related to the occurrence of fattening of $\ol{h}$-level sets. Here follows a result about when this finite-time stabilization takes place. 

\begin{prop}\label{cor:finite-infinite}  Assume that (A1)--(A4) hold and there exists $\ol{h}\in \R$ satisfying \eqref{h-threshold}. Let $u$ be the unique solution  of \eqref{E1} and \eqref{initial}. 
\begin{enumerate}[{\rm(a)}]
\item For $x \in \mathbb{R}^n$, there exists $T>0$ such that 
\begin{equation}\label{finite-convergence} 
u(x,t) = \overline{h} \quad \text{for all $t \ge T$}  
\end{equation}
if either (i) $u_0(x)<\ol{h}$ and \eqref{fattening suff1} hold or (ii) $u_0(x)> \ol{h}$ and \eqref{fattening suff2} hold.
\item  Assume in addition that $f(\overline{h}, \cdot)$ is of class ${\rm Lip}_{\rm loc}([0, \infty))$.  Then, 
$u(x,t) < \overline{h}$ holds for any $t > 0$ and any $x \in \mathbb{R}$ satisfying $u_0(x) < \overline{h}$ provided that \eqref{type2d} hlds. 
Similarly, $u(x,t) > \overline{h}$ holds for any $t > 0$ and any $x \in \mathbb{R}$ satisfying $u_0(x) > \overline{h}$ 
provided that \eqref{type2e} holds. 
\end{enumerate}
\end{prop}

\begin{proof}
For (a), we only prove \eqref{finite-convergence} under the condition (ii). The proofs under (i) is analogous. 
When (ii) holds, by continuity we have $f(\overline{h}, \mu(E_{\ol{\vep}, \vep}))> 0$ for $\vep>0$ small, where we denote $E_{\ol{h}, \vep}:=E_{\overline{h}}(\delta_E(\varepsilon, \overline{h}))$. 
By (A2) and the monotonicity of $t \mapsto \delta_E(t,\overline{h})$, we thus have 
\[
(\delta_E)_t(t, \overline{h}) \ge f({\overline{h}},  \mu (E_{\ol{h}, \vep})) > 0
\] 
for all $t \ge \varepsilon$. Therefore, letting 
\[ T := \frac{{\rm dist}(x, E_{\overline{h}}(0))}{ f({\overline{h}},  \mu (E_{\ol{h}, \vep}))} + \varepsilon, \]
we can easily see that $x \in E_{\overline{h}}(\delta_E(t, \overline{h}))$ for all $t \ge T$, which is equivalent to  \eqref{finite-convergence}. 

For (b), let us also give a proof only in the case $u_0(x)>\overline{h}$ and omit the other case.  
Following the proof of Theorem \ref{prop:fattening-o-h-1}, we have $\delta_E(t,\overline{h}) = 0$ for any $t \ge 0$.  Then for any $x\in \R^n$ satisfying $u_0(x)>\overline{h}$, we get $x \not\in E_{\overline{h}}(\delta_E(t,\overline{h}))$ and thus $u(x,t)>\overline{h}$ for all $t\geq 0$. 
\end{proof}

\subsection{Lower bound for the convergence rate}
As a further step than the result in Proposition \ref{cor:finite-infinite}(b) for Type II behavior, under stronger assumptions we give a more precise convergence rate of \eqref{eq:conv}, as stated in Theorem \ref{thm exponent}. 
Since $u_0$ is assumed to be uniformly continuous, we can find a strictly increasing $g_+\in C([0, \infty))$ and a strictly decreasing function $g_-\in C([0, \infty))$ satisfying $g_\pm(0)=0$, and 
\begin{equation}\label{ini-barrier1}
g_+(\dist(\cdot, E_{\ol{h}}(0)))+\ol{h}\geq u_0 \quad \text{in $\R^n$,} 
\end{equation}
\begin{equation}\label{ini-barrier2}
g_-(\dist(\cdot, D_{\ol{h}}(0)^c))+\ol{h}\leq u_0  \quad \text{in $\R^n$.} 
\end{equation}

We use $g_-$ to construct a barrier function from below and obtain the following result on the convergence rate. 

\begin{prop}\label{prop exponent-lower}
Assume that (A1)--(A4) hold and there exists $\ol{h}\in \R$ satisfying \eqref{h-threshold}. Let $u$ be the unique solution to \eqref{E1} and \eqref{initial}. 
Assume that 
\eqref{reg initial lower} holds. 
Let $g_-\in C([0, \infty))$ be a strictly decreasing function with $g_-(0)=0$ such that \eqref{ini-barrier2} holds. 
Then, there exist $C>0$ and $\lambda>0$ such that
 \begin{equation}\label{exp est2}
u(x, t)-\ol{h} \geq g_-(Ce^{-\lambda t}) \quad \text{for all $x\in \R^n$ and $t>0$ large.}
\end{equation}
 \end{prop}

\begin{proof}
Since $D_{\ol{h}}(0)$ is bounded, we have 
\[
\ul{d}:=\sup_{\R^n} \dist(x, D_{\ol{h}}(0)^c )<\infty.
\] 
By (A2), \eqref{h-threshold} and \eqref{reg initial lower}, we can choose
$\lambda>0$ such that 
\begin{equation}\label{eq exp2 new}
f(\ol{h}, \mu(D_{\ol{h}}(s)))\leq \lambda s \quad\text{for all $-\ul{d}\leq s\leq 0$.}
\end{equation}
For $(x, t)\in \R^n\times [0, \infty)$, take 
\begin{equation}\label{barrier2}
w_-(x, t)=g_-(\dist(x, D_{\ol{h}}(0)^c)e^{-\lambda t})+\ol{h}.
\end{equation}
Fix $h\leq  \ol{h}$ arbitrarily. Since \eqref{ini-barrier2} holds, we have 
\[
D^w_h(0):=\{w_-(\cdot, 0)< h\}\supset \{u_0< h\}.
\]
This relation amounts to saying that $D_h(0)\subset D_{\ol{h}}(-g_-^{-1}(h-\ol{h}))$, since by the definition of $D_h$ and the strict monotonicity of $g_-$, we have $D_h(0)=\{u_0<h\}$ and $\{w_-(\cdot, 0)< h\}=D_{\ol{h}}(-g_-^{-1}(h-\ol{h}))$.

Let us consider the parallel set of $h$-sublevel set of $w_-(\cdot, 0)$: for $s\leq 0$, 
\[
D^w_h(s):=\{x\in \R^n: \dist(x, \{w_-(\cdot, 0)\geq h\})> -s\}= D_{\ol{h}}(-g_-^{-1}(h-\ol{h})+s). 
\]
The second equality above again follows from the definition of $D_h$. Indeed, noticing that $\{w_-(\cdot, 0)\geq h\}=D_{\ol{h}}(-g_-^{-1}(h-\ol{h}))^c$, we see that $\dist(x, \{w_-(\cdot, 0)\geq h\})>-s$ if and only if $\dist(x, D_{\ol{h}})>-g_-^{-1}(h-\ol{h})+s$, which immediately yields the desired equality. 

The ODE \eqref{dist ODE} with $W_h= D^w_h$ writes 
\[
\delta_t(t, h)= (f_h\circ \mu)(D^w_h(\delta(t, h))) \quad \text{for $t>0$ with $\delta(0, h)=0$.}
\]
Taking $\tilde{\delta}(t, h)= \delta(t, h)-g_-^{-1}(h-\ol{h})$,   we can rewrite the equation as 
\[
\tilde{\delta}_t(t, h)=(f_h\circ \mu)(D_{\ol{h}}(\tilde{\delta}(t,h))), \quad \tilde{\delta}(0, h)= -g_-^{-1}(h-\ol{h}).
\]
By (A2) and \eqref{eq exp2 new}, we have 
\[
\tilde{\delta}_t(t, h)\leq (f_{\ol{h}}\circ \mu) (D_{\ol{h}}(\tilde{\delta} (t,h)))\leq \lambda\tilde{\delta}(t,h).
\]
It follows that $\tilde{\delta}(t, h)\leq  -g_-^{-1}(h-\ol{h}) e^{\lambda t}$, which yields
\[
\begin{aligned}
\{\dist(\cdot, D_{\ol{h}}(0)^c)> g_-^{-1}(h-\ol{h}) e^{\lambda t}\}&\supset\{\dist(\cdot, D_{\ol{h}}(0)^c)> -\delta(t, h)+g_-^{-1}(h-\ol{h})\}
\\
&=D_{\ol{h}}(\delta(t, h)-g_-^{-1}(h-\ol{h})). 
\end{aligned}
\]
Noticing that, for $w_-$ defined by \eqref{barrier2} and $h\leq \ol{h}$,  
\[
\{w_-(\cdot, t)< h\}=\{\dist(\cdot, D_{\ol{h}}(0)^c)> g_-^{-1}(h-\ol{h}) e^{\lambda t}\},
\]
we therefore get $D^w_h(\delta(t, h))\subset \{w_-(\cdot, t)< h\}$.

On the other hand, comparing the evolutions for $D_h^w$ and $D_h$, we have ${\delta}(t, h) \geq \delta_D(t, h)$ and $D_h(\delta_D(t, h))\subset D_h^w(\delta(t, h))$ for all $t\geq 0$ and $h\leq \ol{h}$. Hence, we are led to 
\[
\{u(\cdot, t)< h\}\subset \{w_-(\cdot, t)< h\}.
\]
Since $h\leq \ol{h}$ is arbitrary, we immediately get $w_-(x, t) \leq u(x, t)$ for all $x\in \R^n$ and $t\geq 0$. Since $D_{\ol{h}}(0)$ is assumed to be bounded, we thus can find $C>0$ such that \eqref{exp est2} holds. 
\end{proof}

We remark that it is possible to adopt an alternative PDE method to prove \eqref{exp est2} under the assumption \eqref{eq exp2 new}. 
In fact, one can show that $w_-$ given by \eqref{barrier2} is a subsolution of \eqref{E1} and then use the comparison principle to get the desired estimate. Let us below give a sketch of proof, assuming that $g_-\in C^1((0, \infty))$.  

Suppose that there exist $\phi\in C^1(\R^n\times (0, \infty))$ and $(x_0, t_0)\in \R^n\times (0, \infty)$ such that $w_--\phi$ attains a local maximum at $(x_0, t_0)$. We may assume that $w_-(x_0, t_0)-\phi(x_0, t_0)=0$. Since $g_-$ is strictly decreasing, it follows that 
\[
(x, t)\mapsto\dist(x, D_{\ol{h}}(0)^c)-g_-^{-1}(\phi(x, t)-\ol{h})e^{\lambda t}
\]
also attains a local minimum at $(x_0, t_0)$. Let us discuss two different cases about the location of $x_0$.

Suppose that $x_0\in D_{\ol{h}}(0)$. Then, we see that
\[
t\mapsto \dist(x_0, D_{\ol{h}}(0)^c)-g_-^{-1}(\phi(x_0, t)-\ol{h})e^{\lambda t} 
\]
attains a local minimum at $t=t_0$, which immediately yields 
\begin{equation}\label{eq exp3}
\phi_t(x_0, t_0)=-\lambda d_0 e^{-\lambda t_0} g_-'(d_0 e^{-\lambda t_0}),
\end{equation} 
where we set $d_0:=\dist(x_0, D_{\ol{h}}(0)^c)$.
Also, since
\[
x\mapsto\dist(x, D_{\ol{h}}(0)^c)-g_-^{-1}(\phi(x, t_0)-\ol{h})e^{\lambda t_0} 
\]
achieves a local minimum at $x=x_0$, we have 
\begin{equation}\label{eq exp5}
|\nabla \phi(x_0, t_0)| \geq |g_-'(d_0 e^{-\lambda t_0}) e^{-\lambda t_0}|=-g_-'(d_0 e^{-\lambda t_0}) e^{-\lambda t_0}.
\end{equation}
 Here we applied the fact that, for any closed set $A\subset \R^n$, $U=\dist(\cdot, A)$ is known to be a viscosity solution of the eikonal equation $|\nabla U|=1$ in $A^c$.
Moreover, noticing that
\[
\begin{aligned}
&f(w_-(x_0, t_0), \mu(\{w_-(\cdot , t_0)< w_-(x_0, t_0)\}))\\
&=f(w_-(x_0, t_0), \mu(\{\dist(\cdot, D_{\ol{h}}(0)^c)> \dist(x_0, D_{\ol{h}}(0)^c)\}))= f(w_-(x_0, t_0), \mu(D_{\ol{h}}(-d_0))),
\end{aligned}
\]
we adopt (A2) and \eqref{eq exp2 new} to obtain
\begin{equation}\label{eq exp6}
f(w_-(x_0, t_0), \mu(\{w_-(\cdot , t_0)\leq w_-(x_0, t_0)\}))\leq f(\ol{h}, \mu(D_{\ol{h}}(-d_0))) \leq -\lambda d_0.
\end{equation}
Now combining \eqref{eq exp3}, \eqref{eq exp5} and \eqref{eq exp6}, we have
\begin{equation}\label{eq exp7}
\phi_t(x_0, t_0) +|\nabla \phi(x_0, t_0)| f(w_-(x_0, t_0), \mu(\{w_-(\cdot , t_0)\leq w_-(x_0, t_0)\}))\leq 0. 
\end{equation}
For the case $x_0\in D_{\ol{h}}(0)^c$, it is easily seen that $\phi_t(x_0, t_0)\leq 0$ and $\{w_-(\cdot , t_0)< w_-(x_0, t_0)\}=D_{\ol{h}}(0)$. By (A2) and \eqref{h-threshold} we have 
\[
f(w_-(x_0, t_0), \mu(\{w_-(\cdot , t_0)< w_-(x_0, t_0)\}))\leq f(\ol{h}, \mu(D_{\ol{h}}(0)))\leq 0.
\]
It follows that \eqref{eq exp7} still holds. 
We have shown that $w_-$ is a subsolution of \eqref{E1} provided that $g_-\in C^1((0, \infty))$. For the general case when $g_-\in C([0, \infty))$, we can approximate $g_-$ locally uniformly by a sequence of such $C^1$ functions. Since we have verified that $w_-$ is a subsolution of \eqref{E1}, using the comparison principle, we immediately get $w_-\leq u$ in $\R^n\times (0, \infty)$, which implies \eqref{exp est2}.

\subsection{Upper bound for the convergence rate}
Following the PDE method mentioned above, one may take in a symmetric manner
\[
w_+(x, t)=g_+(\dist(x, E_{\ol{h}}(0))e^{-\lambda t}))+\ol{h}, \quad\text{for $(x, t)\in \R^n\times [0, \infty)$, $\lambda>0$,}
\]
and attempt to show that $w_+$ is a supersolution of \eqref{E1}. 
However, for the argument to work in this case we need to find $\lambda>0$ such that
$f(\ol{h}, \mu(E_{\ol{h}}(s)))\geq \lambda s$ for all $s\geq 0$,
which fails to hold, given the boundedness of $f$. Hence, we continue to adopt the control-theoretic interpretation to look into the upper bound of $u(x, t)-\ol{h}$. In this case, we impose an additional assumption on the growth rate of dynamics $f$ near the critical level set. 

\begin{prop}\label{prop exponent-upper}
Assume that (A1)--(A4) hold and there exists $\ol{h}\in \R$ satisfying \eqref{h-threshold}. Let $u$ be the unique solution to \eqref{E1} and \eqref{initial}. Assume that \eqref{reg initial} holds. 
Let $g_+\in C([0, \infty))$ be a strictly increasing function with $g_+(0)=0$ such that \eqref{ini-barrier1} holds. Then, for any compact set $K\subset \R^n$, there exist  $\lambda>0$ and $C_K>0$ such that 
\begin{equation}\label{exp est3}
u(x, t)-\ol{h}\leq g_+\left(C_Ke^{-\lambda t}\right) \quad \text{for all $x\in K$ and $t> 0$ large.}
\end{equation}
\end{prop}

\begin{proof}
Let $K$ be a compact subset of $\R^n$. We can choose $b>0$ small such that 
\begin{equation}\label{reg initial2}
\ol{d}:= \sup_{x\in K} {\rm dist}(x, E_{\overline{h}}(0)) + 1\leq 1/b.
\end{equation}


Let $\delta_E$ be the maximal solution of \eqref{dist ODE} with $W_h=E_h$ for $h\geq \ol{h}$.
Applying \eqref{reg initial} to 
 \eqref{dist ODE} and using the simplified notation $E_{h, t}$ as in \eqref{abbrev}, we have, with $\lambda=a$ in \eqref{reg initial},  
\begin{equation}\label{pre gronwall}
(\delta_E)_t(t, h)= f(h, \mu(E_{h, t}))\geq f(h, \mu(E_{h, 0})) +\lambda \delta_E(t, h) 
\end{equation}
for any  $t > 0$ and $\overline{h} \leq h \le \overline{h} + b$ satisfying $\delta_E(t,h) \le \ol{d}$. 
This yields
\begin{equation}\label{gronwall barrier1}
\delta_E(t, h) \ge {1\over \lambda} f(h, \mu(E_h(0))) (e^{\lambda t} -1)
\end{equation}
for $h, t$ satisfying the same conditions. Since $u_0\leq w_+(\cdot, 0)$ in $\R^n$ implies 
$E_{\ol{h}}(s_h)\subset E_h(0)$ for $s_h:=g_+^{-1}(h-\ol{h})$, by (A2) and \eqref{reg initial} again, we have
\begin{equation}\label{gronwall barrier2}
f(h, \mu(E_h(0)))\geq f(\ol{h}, \mu(E_{\ol{h}}(s_h)))\geq \lambda s_h.
\end{equation}
Taking 
\begin{equation}\label{h close}
h(t)=g_+\left(\frac{1}{e^{\lambda t}-1}\sup_{x\in K}{\dist (x, E_{\ol{h}}(0))}\right)+\ol{h},
\end{equation}
we have $\ol{h}\leq h(t)\leq \ol{h}+b$ when $t>0$ is large. Also, for $x\in K$ and sufficiently large $t>0$, by \eqref{gronwall barrier1}, \eqref{gronwall barrier2} and \eqref{h close} we get
\begin{equation}\label{h-asymptotics}
\delta_E(t, h(t))\geq \dist(x, E_{\ol{h}}(0))\geq \dist(x, E_{h(t)}(0))
\end{equation}
provided that $\delta_E(t, h(t))\leq \ol{d}$ holds. If $\delta_E(t, h(t))> \ol{d}$, then in view of \eqref{reg initial2} we still have $\delta_E(t, h(t))\geq \dist(x, E_{h(t)}(0))$. In both cases, we get $K\subset E_{h(t)}(\delta_E(t, h(t)))$ when $t>0$ is large. Adopting \eqref{coincide1} in Proposition \ref{prop:parallel-surfaces}, we obtain
\[
u(x, t)\leq h(t)=g_+\left(\frac{1}{e^{\lambda t}-1}\sup_{x\in K}{\dist (x, E_{\ol{h}}(0))}\right)+\ol{h} \quad \text{for all $x\in K$.}
\]
We thus can choose $C_K>0$ such that \eqref{exp est3} holds. 
\end{proof}

\begin{rem}\label{rmk improved decay}
In some good cases,  
we can obtain \eqref{exp est3} with $\lambda>0$ independent of the compact set $K$. For example, if, instead of \eqref{reg initial}, there exist $\lambda>0$ and  $0<\beta\leq 1$ such that, 
\begin{equation}\label{reg initial change}
f(h, \mu(E_{h}(s)))\geq f(h, \mu(E_{h}(0)))+{\lambda s\over 1+s^\beta}\quad \text{for all $\ol{h}\leq h\leq \ol{h}+b$ and $s>0$,}
\end{equation}
then in the proof above \eqref{pre gronwall} can be substituted with 
\[
(\delta_E)_t(t, h)\geq f(h, \mu(E_{h, 0})) +{\lambda \delta_E(t, h) \over 1+\delta_E(t, h)^\beta}\geq  {\lambda \left(\delta_E(t, h)+{1\over \lambda}f(h, \mu(E_{h, 0}))\right) \over 1+(\delta_E(t, h)+{1\over \lambda}f(h, \mu(E_{h, 0})))^\beta},
\]
which yields 
\[
{1\over \beta}\left(\delta_E(t, h)+{1\over \lambda}f(h, \mu(E_{h, 0}))\right)^\beta+ \log \left(\frac{\lambda \delta_E(t, h)+f(h, \mu(E_{h, 0}))}{f(h, \mu(E_{h, 0}))}\right)\geq {1\over \beta}\left({1\over \lambda}f(h, \mu(E_{h, 0}))\right)^\beta+\lambda t.
\]
Then there exists $C>0$ such that
\[
C \delta_E(t, h)^\beta \geq \lambda t + \log f(h, \mu(E_{h, 0}))
\]
holds for all $t>0$ large and $\ol{h}\leq h\leq \ol{h}+b$. Note that in place of \eqref{gronwall barrier2} we now have
\[
f(h, \mu(E_h(0)))\geq f(\ol{h}, \mu(E_{\ol{h}}(s_h)))\geq {\lambda s_h\over  1+s_h^\beta}.
\]     
We then obtain \eqref{h-asymptotics} for all $x\in K$ and $t>0$ sufficiently large by letting 
\[
h(t)=\ol{h}+g_+(A e^{C{d}_K^\beta-\lambda t})
\]
with ${d}_K:=\sup_{x\in K}\dist (x, E_{\ol{h}}(0))$ and some $A>1/\lambda$. It thus follows that 
\begin{equation}\label{upper decay}
\sup_{x\in K} u(x, t)-\ol{h}\leq g_+(A e^{C{d}_K^\beta-\lambda t})\quad \text{for all $t>0$ large. }
\end{equation}
The exponent $\lambda$ of the exponential decay in time is independent of $K$. 
\end{rem}

 Theorem \ref{thm exponent} is an immediate consequence of Proposition \ref{prop exponent-lower} and Proposition \ref{prop exponent-upper}. 
 We can obtain \eqref{exp est2} and \eqref{exp est3} with $g_\pm$ replaced by the modulus of continuity $g$ of $u_0$.

Let us further discuss the assumption \eqref{reg initial}. It can actually be understood as a combined condition on the regularity of $u_0$ and the non-degeneracy of $f$ in the measure argument. 

\begin{prop}\label{prop reg initial}
Assume that (A1)--(A4) hold and there exists $\ol{h}\in \R$ satisfying \eqref{h-threshold}. Assume that for any bounded set $A\subset \R^n$, there exists $\Theta_{A}> 0$ such that \eqref{eq a5} holds. 
Assume that 
\begin{equation}\label{non-degeneracy}
f(\overline{h}, \cdot) \in C^1([0, \infty)), \quad f_q(\overline{h}, \mu(E_{\overline{h}}(0)))>0.
\end{equation}
Assume in addition that $u_0$ is quasiconvex and $E_{\ol{h}}(0)=\{u_0\leq \ol{h}\}$ has nonempty interior. Then \eqref{reg initial} holds. 
\end{prop}

Under the convexity condition of level sets, we can utilize the Steiner formula, which provides a representation of $m(D_h(s))$ and $m(E_h(s))$ in terms of a polynomial of degree $n$ for $s > 0$ with geometrical coefficients determined by $D_h(0)$ and $E_h(0)$, respectively.  For the reader's convenience, we present the formula in a form applicable to our problem and refer to \cite[Theorem 5.6--Remark 5.10]{Fe} and \cite[Theorem 3.2.35]{Fe2} for more general results and proofs.


\begin{prop}\label{prop:steiner}
Let $A \subset \mathbb{R}^n$ be a bounded convex body, that is, $A$ is convex with nonempty interior. 
There exist constants $\Phi_i(A)$ for $i=0, 1, 2, \cdots n$ such that $\Phi_0(A) = m(A)$, $\Phi_1(A) = {\rm Per}(A)$ and 
\[ m(\{x \in \mathbb{R}^n: {\rm dist}(x, A) \le s\}) = \sum_{i=0}^n \Phi_i(A) s^i \quad \text{for} \; \; s \ge 0. \]
Furthermore, if a family of bounded convex bodies $A_j \subset \mathbb{R}^n$ converges to a bounded convex body $A \subset \mathbb{R}^n$ in the Hausdorff metric, then 
\[ \lim_{j \to \infty} \Phi_i(A_j) = \Phi_i(A) \quad \text{for} \; \; i=0, 1,2, \cdots, n. \]
\end{prop}

When $u_0$ is quasiconvex in $\R^n$, by our control-theoretic representation formulas \eqref{new upper value} and \eqref{new lower value}, we can easily see that $u(\cdot, t)$ is quasiconvex for all $t>0$. 
We refer to \cite{C, KLM} for more general results on quasiconvexity preserving property for 
nonlocal parabolic equations. 

 Moreover, if in addition
\begin{equation}\label{non-fattening-convex}
(f_{\overline{h}} \circ \mu)(D_{\overline{h}}(0)) = (f_{\overline{h}} \circ \mu)(E_{\overline{h}}(0)) = 0
\end{equation}
holds, then
it immediately follows from Proposition \ref{cor:finite-infinite}(b) that $u(x, t)>\ol{h}$ for any $x\in \R^n$ satisfying $u_0(x)>\ol{h}$,  since the perimeters of the parallel sets are bounded due to the convexity and boundedness of the sublevel sets of $u_0$.


\begin{lem}\label{lem:con-coe}
Assume that (A1)--(A4) hold.  Assume that $u_0$ is quasiconvex in $\R^n$. 
Let $\overline{h}$ satisfy \eqref{h-threshold} with $D_{\ol{h}}(0)\neq \emptyset$. If \eqref{non-fattening-convex} holds, then
\[ \lim_{h \to \overline{h}+} \Phi_i(D_h(0)) = \Phi_i(D_{\overline{h}}(0)), \; \; \lim_{h \to \overline{h}+} \Phi_i(E_h(0)) = \Phi_i(E_{\overline{h}}(0)) \quad \text{for all} \ i=0, 1, 2, \cdots, n,
 \]
where $\Phi_i$ given in Proposition {\rm\ref{prop:steiner}}. 
\end{lem}

\begin{proof}
In view of Proposition \ref{prop:steiner}, it is sufficient to prove that $D_h(0)$ and $E_h(0)$ converge to $D_{\overline{h}}(0)$ and $E_{\overline{h}}(0)$, respectively, as $h \to \overline{h}+$ in the Hausdorff metric. 
To show the convergence of $D_h(0)$, we assume by contradiction that there exist $\varepsilon > 0$, $h_j > \overline{h}$ and $x_j \in D_{h_j}(0)$ such that $\lim_{j \to \infty} h_j = \overline{h}$ and ${\rm dist}(x_j, D_{\overline{h}}(0)) \ge \varepsilon$  for any $j \in \mathbb{N}$. 
Then, by taking a subsequence we have $x_j\to x$ as $j\to\infty$ for some $x\in E_{\overline{h}}(0)$ satisfying 
\begin{equation}\label{x-dist} 
{\rm dist}(x, D_{\overline{h}}(0)) \ge \varepsilon. 
\end{equation}
Choose an open ball $B \subset D_{\overline{h}}(0)$ and define a cone $C(B, x)$ as 
\[ C(B,x) := \{ \lambda x + (1-\lambda) y: \lambda \in [0,1], \; \; y \in B\}. \]
We then see $C(B,x) \subset E_{\overline{h}}(0)$ due to the convexity of $E_{\overline{h}}(0)$. 
Furthermore, by \eqref{x-dist} we have
\[ \mu(E_{\overline{h}}(0)) \ge \mu(C(B,x) \cap B_{\varepsilon/2}(x)) + \mu(D_{\overline{h}}(0)), \]
which contradicts \eqref{non-fattening-convex} since $\mu(C(B,x) \cap B_{\varepsilon/2}(x))$ is obviously positive and $f$ is strictly increasing. 
The convergence of $E_h(0)$ to $E_{\overline{h}}(0)$ can be proved similarly. 
\end{proof}

Let us now prove Proposition \ref{prop reg initial}.

\begin{proof}[Proof of Proposition \ref{prop reg initial}]
By \eqref{non-degeneracy}, we can find $L>0$ and let $\vep> 0$ small such that 
\begin{equation}\label{min-max-f} 
\begin{aligned}
& f(h, q_2) - f(h, q_1) \ge L (q_2 - q_1) \\
&\text{for} \; \;  \mu(E_{\overline{h}}(0))\leq q_1\leq q_2 \leq \mu(E_{\overline{h}}(0))+\vep, \; \overline{h}\le h \le \overline{h} + \varepsilon. 
\end{aligned}
\end{equation}
By (A3), \eqref{eq a5},  Proposition \ref{prop:steiner}, Lemma \ref{lem:con-coe} as well as the monotonicity of $s \mapsto m(D_h(s))$, for any $b>0$ and $0<\vep<b$, there exists $0<c<b\vep$ such that 
\begin{equation}\label{min-max-m} 
 \mu(E_h(s)) - \mu(E_h(0))\geq cs  \quad \text{for} \; \; \overline{h} \le h \le \overline{h} + \varepsilon, \; \; 0 \le s \le 1/b. 
\end{equation}

Applying \eqref{min-max-f} with $q_1=\mu(E_h(0))$ and $q_2=\mu(E_h(0))+cs$ ($0\leq s\leq 1/b$), we utilize \eqref{min-max-m} and (A2) to deduce 
\[
f(h, \mu(E_h(s)))-f(h,  \mu(E_h(0)))\geq Lc s,
\]
which verifies \eqref{reg initial} with $a=Lc$.
\end{proof}

We conclude this work with a concrete example, which shows that the exponential-like convergence rate shown in Theorem \ref{thm exponent} is optimal. 
\begin{ex}\label{ex decay}
Let us take $f(r, q)=F(q)$, where $F: [0, \infty)\to \R$ is a bounded, continuous, strictly increasing function such that $F(q)=q-1$ for all $0\leq q\leq 2$. We consider \eqref{E1}\eqref{initial} with $n=2$. Let the density function $\Theta$ in (A3) be  given by 
\[
\Theta(x)={1\over \pi (1+|x|)^2}, \quad 
\]
which yields, for the disk $B_\rho\subset \R^2$ centered at the origin with radius $\rho\geq 0$,  
\[
\mu(B_\rho)=\int_{B_\rho} {1\over \pi(1+|x|)^2}\, dx=\int_0^\rho \int_0^{2\pi} {1\over \pi(1+\sigma)^2}\, ds\, d\sigma={2\rho\over \rho+1}\in [0, 2).
\]
Suppose that the initial value $u_0$ is radially symmetric with respect to the origin, i.e., we can write $u_0(x)=v_0(|x|)$ for a function $v_0\in C([0, \infty))$. We further assume that $v_0$ is strictly increasing. In view of our optimal control interpretation, we see that the unique solution $u$ preserves the spatial symmetry and the strict monotonicity in the radial variable for all times. We thus can express $u(x, t)=v(|x|, t)$ in terms of a continuous function $v\in C([0, \infty)\times [0, \infty))$ with $\rho\mapsto v(\rho, t)$ strictly increasing. It follows that, for any $(x, t)\in \R^2\times [0, \infty)$ with $\rho=|x|$,  
\[
\mu(\{u(\cdot, t)< u(x, t)\})=\mu(\{u(\cdot, t)\leq u(x, t)\})={2\rho\over \rho+1}, 
\]
which implies that 
\[
f(u(x, t), \mu(\{u(\cdot, t)< u(x, t)\}))=f(u(x, t), \mu(\{u(\cdot, t)\leq u(x, t)\}))={\rho-1\over \rho+1}.
\]
Then in this case $\ol{h}=v_0(1)$ and \eqref{E1} reduces to the state constraint problem:
\[
\left\{
\begin{aligned}
&v_t+|v_\rho| {\rho-1\over \rho+1}=0 \quad \text{in $(0, \infty)\times (0, \infty)$,}\\
&v_t+|v_\rho| {\rho-1\over \rho+1}\geq 0 \quad \text{in $[0, \infty)\times (0, \infty)$.}
\end{aligned}
\right.
\]
Using the standard control-based representation formula for viscosity solutions of (local) Hamilton-Jacobi equations (for example,  \cite{BCBook}), we see that the unique viscosity solution can be expressed by
\begin{equation}\label{eq loc control}
v(\rho, t)=\begin{cases}
\inf_{\alpha} v_0(y(t; \alpha))& \text{if $\rho\geq 1$,}\\
\sup_{\alpha} v_0(y(t; \alpha)) & \text{if $0\leq \rho<1$,}
\end{cases}
\end{equation}
where $\alpha\in L^\infty(0, \infty))$ denotes a control function satisfying $|\alpha(s)|\leq 1$ a.e. and $y: [0, \infty)\to [0, \infty)$ is a locally Lipschitz function satisfying the state equation 
\[
y'(s)= \alpha(s)\frac{y(s)-1}{y(s)+1}, \quad s>0 \quad \text{with $y(0)=\rho$.}
\]
The optimal control is clearly $\alpha\equiv -1$ and a direct computation yields 
\[
\log(e^{y(t)} (y(t)-1)^2) -\log (e^\rho (\rho-1)^2) =\int_0^t \alpha(s)\, ds=-t.
\]
It follows from \eqref{eq loc control} that $v(\rho, t)=v_0(\rho(t))$, where $\rho(t)\geq 0$ satisfies 
\[
e^{\rho(t)}(\rho(t)-1)^2=(\rho-1)^2e^{\rho-t}.
\]
Since $e^{\rho(t)}\geq 1$, we thus have $|\rho(t)-1|\leq |\rho-1| e^{\rho-t\over 2}$.
Utilizing the modulus of continuity $g$ of $u_0$, we are led to 
\[
|u(x, t)-\ol{h}|=|v_0(\rho(t))-v_0(1)|\leq g(|\rho(t)-1|)\leq g\left(||x|-1| e^{|x|-t\over 2}\right)
\]
for all $x\in \R^2$ and $t\geq 0$. This estimate is consistent with \eqref{large time est} in Theorem \ref{thm exponent}. It actually agrees with the case disucssed in Remark \ref{rmk improved decay}; the exponent $1/2$ for the exponential decay in time does not depend on $x$. Note that in this example, for any $0<\lambda<1/2$ close to $1/2$, there exists $b>0$ small such that 
\[
f(h, \mu(E_h(s)))=\frac{v_0^{-1}(h)+s-1}{v_0^{-1}(h)+s+1}
\]
satisfies \eqref{reg initial change} with $\beta=1$. Our result \eqref{upper decay} in Remark \ref{rmk improved decay}, together with Proposition \ref{prop exponent-lower}, ensures that for any compact set $K\subset \R^2$, there exist $A, C>0$ such that %
\[
\sup_{x\in K}|u(x, t)-\ol{h}|\leq \sup_{x\in K} g(A e^{C(|x|+1)-\lambda t})
\]
for any $t>0$ large and  any $\lambda<1/2$. However, the exponent $1/2$ is not covered. It would be interesting to further establish general convergence results for the optimal exponent. 
\end{ex}

\appendix

\section{Proof of comparison principle}\label{sec:app}

In this appendix we present a proof of Theorem \ref{thm:comparison1}.

\begin{proof}[Proof of Theorem \ref{thm:comparison1}]
Assume by contradiction that $\sup_{\R^n \times [0,T)}(u-v) =: \theta >0$. 
Then, there exists $\lambda > 0$ such that 
\[ \sup_{(x,t) \in \R^n \times [0,T)} \left\{u(x,t) - v(x,t) - \frac{\lambda}{T-t} \right\} > \frac{3 \theta}{4}. \]
There exists $(x_1, t_1) \in \R^n \times [0,T)$ such that $u(x_1, t_1) - v(x_1, t_1) - \lambda/ (T-t_1) > \theta/2$. 
Since $\sup_{\R^n}(u(\cdot,0) - v(\cdot,0)) \le 0$, we have $t_1 > 0$. Define 
\[ \Phi(x,y,t) := u(x,t) - v(y,t) - \frac{|x-y|^2}{\varepsilon^2} - \alpha (|x|^2 + |y|^2) - \frac{\lambda}{T-t} \]
for $\varepsilon, \alpha > 0$. 
It is then clear that there exists $\alpha_0 > 0$ small such that 
\begin{equation}\label{tilde_sup} 
\sup_{(x, y, t) \in \R^{2n} \times [0, T)} \Phi (x, y, t) > \frac{\theta}{4} 
\end{equation}
for all $0< \alpha < \alpha_0$ and $\varepsilon >0$ small. 
By the same argument as that in proof of \cite[Theorem~3.1]{KLM}, 
we can prove that for any fixed $T>0$ and $M_1 > M_T$ large, there exists $M_2 > 0$ such that 
\begin{equation}\label{ineq:growth}
u(x,t) - v(y,t) \le M_1 |x-y| + M_2 (1+t) \quad \text{for} \; \; x, y \in \R^n \; \; \text{and} \; \; t \in [0,T),  
\end{equation}
which implies that $\Phi$ attains a maximum at some $(x_{\ep, \alpha}, y_{\ep, \alpha}, t_{\ep, \alpha}) \in \R^{2n} \times [0,T)$. 
We write $(\tilde{x}, \tilde{y}, \tilde{t})$ for $(x_{\ep, \alpha}, y_{\ep, \alpha}, t_{\ep, \alpha})$ to simplify our notations. 
It follows that 
\begin{align*}
&\frac{|\tilde{x} - \tilde{y}|^2}{\ep^2} + \alpha (|\tilde{x}|^2 + |\tilde{y}|^2) \le u(\tilde{x}, \tilde{t}) - v(\tilde{y}, \tilde{t}) - u(x_1, t_1) + v(y_1, t_1) + 2\alpha |x_1|^2 + \frac{\lambda}{T- t_1} - \frac{\lambda}{T - \tilde{t}}. 
\end{align*}
In view of \eqref{ineq:growth}, we have 
\begin{align*}
&\frac{|\tilde{x} - \tilde{y}|^2}{\ep^2} + \alpha(|\tilde{x}|^2 + |\tilde{y}|^2) \\
&\le M_1|\tilde{x} - \tilde{y}| + M_2(\tilde{t}+1) - u(x_1, t_1) + v(x_1, t_1) + 2\alpha |x_1|^2 + \frac{\lambda}{T-t_1} - \frac{\lambda}{T- \tilde{t}}. 
\end{align*}
It follows that 
\[ \frac{|\tilde{x} - \tilde{y}|^2}{\ep^2} - M_1 |\tilde{x} - \tilde{y}| + \alpha (|\tilde{x}|^2 + |\tilde{y}|^2) \le C \]
for some $C \ge 0$ which is independent of $\ep, \alpha$, which implies that $\sup_{0 < \alpha < \alpha_0} |\tilde{x} - \tilde{y}| \to 0$ as $\ep \to 0$ and, for any $\vep>0$, $\alpha(|\tilde{x}| + |\tilde{y}|) \to 0$ as $\alpha \to 0$. We also have
\[ 
\sup_{0 < \ep < \ep_0, \ 0 < \alpha < \alpha_0} u(\tilde{x}, 0 ) - v(\tilde{y}, 0) \le \omega_0(\ep_0) \le \theta/4 \]
where $\omega_0$ is the modulus of continuity appearing in \eqref{initial_modulus}. 
Due to \eqref{tilde_sup}, we have $\tilde{t} > 0$ for all $0 < \ep < \ep_0$ and $0 < \alpha < \alpha_0$. 
Here, we fix $\ep > 0$ small enough so that $\tilde{t} > 0$. In what follows, we discuss two cases.

Case 1: Suppose that $\liminf_{\alpha \to 0} |\tilde{x} - \tilde{y}| > 0$ holds. Since $\Phi(x,x,\tilde{t}) \le \Phi(\tilde{x}, \tilde{y}, \tilde{t})$ for all $x \in \R^n$, we have 
\[ u(x, \tilde{t}) - u(\tilde{x}, \tilde{t}) \le v(x, \tilde{t}) - v(\tilde{y}, \tilde{t}) - \frac{|\tilde{x} - \tilde{y}|^2}{\ep^2} + 2 \alpha |x|^2 - \alpha(|\tilde{x}|^2 + |\tilde{y}|^2), \]
which implies $\tilde{V}_\alpha \subset U_\alpha$, where 
\[ U_\alpha := \{u(\cdot, \tilde{t}) < u(\tilde{x}, \tilde{t})\}, \quad \tilde{V}_\alpha := \left\{v(\cdot, \tilde{t}) + 2 \alpha |\cdot|^2 \le v(\tilde{y}, \tilde{t}) +  \frac{|\tilde{x} - \tilde{y}|^2}{\ep^2} \right\}. \]
Notice also that $u(\tilde{x}, \tilde{t}) > v(\tilde{y}, \tilde{t})$ due to $\Phi(\tilde{x}, \tilde{y}, \tilde{t}) > \theta/4$. 
Letting $V_\alpha := \{v(\cdot, \tilde{t}) \le v(\tilde{y}, \tilde{t}) \}$,  
then by (A1)--(A3) we have 
\begin{align*} 
&\; f(v(\tilde{y}, \tilde{t}), \mu(V_\alpha)) - f(u(\tilde{x}, \tilde{t}), \mu(U_\alpha)) \le\; f(u(\tilde{x}, \tilde{t}), \mu(V_\alpha)) - f(u(\tilde{x}, \tilde{t}), \mu(U_\alpha)) \\
=&\; f(u(\tilde{x}, \tilde{t}), \mu(V_\alpha)) - f(u(\tilde{x}, \tilde{t}), \mu(V_\alpha \cap \tilde{V_\alpha})) + f(u(\tilde{x}, \tilde{t}), \mu(V_\alpha \cap \tilde{V}_\alpha)) - f(u(\tilde{x}, \tilde{t}), \mu(U_\alpha)) \\
\le&\; \omega_f (\mu(V_\alpha \setminus \tilde{V}_\alpha)). 
\end{align*}
Due to the assumption $\liminf_{\alpha \to 0} |\tilde{x} - \tilde{y}| > 0$, we have 
\begin{align*}
\bigcap_{0 < \alpha <\alpha_0} \bigcup_{0<\tilde{\alpha}<\alpha} V_{\tilde{\alpha}} \setminus \tilde{V}_{\tilde{\alpha}} = \bigcap_{0 < \alpha <\alpha_0} \bigcup_{0<\tilde{\alpha}<\alpha} \left\{v(\tilde{y}, \tilde{t}) + \frac{|\tilde{x} - \tilde{y}|^2}{\ep^2} - 2\tilde{\alpha} |\cdot|^2 < v(\cdot, \tilde{t}) \le v(\tilde{y}, \tilde{t}) \right\} = \emptyset, 
\end{align*}
which implies 
\[ \limsup_{\alpha \to 0} \mu(V_\alpha \setminus \tilde{V}_\alpha) \le \mu\left( \bigcap_{0 < \alpha <\alpha_0} \bigcup_{0<\tilde{\alpha}<\alpha} V_{\tilde{\alpha}} \setminus \tilde{V}_{\tilde{\alpha}}\right) = \mu(\emptyset) = 0. \]
Therefore it follows that 
\begin{equation}\label{set-nonnega}
\limsup_{\alpha \to 0} f(v(\tilde{y}, \tilde{t}), \mu(V_\alpha)) - f(u(\tilde{x}, \tilde{t}), \mu(U_\alpha)) \le 0. 
\end{equation}

Since $u$ and $v$ are respectively a  subsolution and supersolution to \eqref{E1}, we get 
\begin{equation}\label{com-subsuper} 
h + |p_1| f(u(\tilde{x}, \tilde{t}), \mu(U_\alpha)) \le 0, \quad k + |p_2| f(v(\tilde{y}, \tilde{t}), \mu(V_\alpha)) \ge 0, 
\end{equation}
where $h, k\in\R$ are real numbers satisfying 
\begin{equation}\label{con-hk} 
h - k = \frac{\lambda}{(T-\tilde{t})^2}, 
\end{equation}
\begin{equation}\label{con-p1p2} 
p_1 = \frac{2 (\tilde{x} - \tilde{y})}{\ep^2} + 2\alpha \tilde{x}, \quad p_2 = \frac{2 (\tilde{x} - \tilde{y})}{\ep^2} - 2\alpha \tilde{y}. 
\end{equation}
Therefore, applying \eqref{set-nonnega}, we have 
\begin{align*}
\frac{\lambda}{T^2} \le&\; \limsup_{\alpha \to 0} \frac{\lambda}{(T-\tilde{t})^2} = \limsup_{\alpha \to 0} \big(h-k\big) \\
\le&\; \limsup_{\alpha \to 0} \Big(|p_1 - p_2| f(u(\tilde{x}, \tilde{t}), \mu(U_\alpha)) + |p_2| \big(f(v(\tilde{y}, \tilde{t}), \mu(V_\alpha)) - f(u(\tilde{x}, \tilde{t}), \mu(U_\alpha))\big)\Big) \le 0, 
\end{align*}
which is a contradiction. 

Case 2: Suppose that there exists a sequence $\alpha_i$ such that $|\tilde{x} - \tilde{y}|\to 0$ as $\alpha_i \to 0$. Note that in this case \eqref{com-subsuper}  still holds with $h, k, p_1, p_2$ satisfying \eqref{con-hk} and \eqref{con-p1p2}. 
We thus obtain 
\[
\frac{\lambda}{T^2} \le 
 \lim_{\alpha_i \to 0} \big(h-k\big) 
\le\; \lim_{\alpha_i \to 0} \big(|p_2| f(v(\tilde{y}, \tilde{t}), \mu(V_\alpha)) - |p_1| f(u(\tilde{x}, \tilde{t}), \mu(U_\alpha))\big) = 0, 
\]
which is a contradiction.  
\end{proof}





\bibliographystyle{abbrv}


\end{document}